\documentclass[12pt,reqno]{amsart}
\usepackage{color}
\usepackage[english]{babel}
\usepackage{amsfonts}
\usepackage{amsmath}
\usepackage{amsthm}
\usepackage{amssymb}
\usepackage[misc]{ifsym}
\usepackage{bbm}
\usepackage{mathrsfs}
\usepackage{esint}
\usepackage{faktor}
\usepackage[utf8]{inputenc}
\usepackage{afterpage}
\usepackage[left=2.9cm,right=2.9cm,top=2.8cm,bottom=2.8cm]{geometry}
\usepackage{graphicx}
\usepackage[pagewise,displaymath,mathlines]{lineno}
\newcommand{\N}{\mathbb{N}}

\newcommand{\R}{\mathbb{R}}

\newcommand{\Div}{\mathrm{div} \, }
\newcommand{\dx}{\, {\rm d} x}
\newcommand{\dt}{\, {\rm d} t}
\newcommand{\ds}{\, {\rm d} s}
\renewcommand{\epsilon}{\varepsilon}
\renewcommand{\phi}{\varphi}

\newtheorem{lemma}{Lemma}[section]
\newtheorem{thm}[lemma]{Theorem}

\theoremstyle{definition}
\newtheorem{defi}[lemma]{Definition}
\newtheorem{rmk}[lemma]{Remark}
\newtheorem{ex}[lemma]{Example}

\numberwithin{equation}{section}

\begin{document}

\title[Singular $\Phi$-Laplacian problems]{Existence of two solutions for \\ singular $\Phi$-Laplacian problems}
\author[P. Candito]{Pasquale Candito}
\address[P. Candito]{Department DICEAM, University of Reggio Calabria, Via Graziella (Feo Di Vito), 89122 Reggio Calabria, Italy}
\email{pasquale.candito@unirc.it}
\author[U. Guarnotta]{Umberto Guarnotta}
\address[U. Guarnotta]{Dipartimento di Matematica e Informatica, Universit\`a degli Studi di Palermo, Via Archirafi 34, 90123 Palermo, Italy}
\email{umberto.guarnotta@unipa.it}
\author[R. Livrea]{Roberto Livrea}
\address[R. Livrea]{Dipartimento di Matematica e Informatica, Universit\`a degli Studi di Palermo, Via Archirafi 34, 90123 Palermo, Italy}
\email{roberto.livrea@unipa.it}

\maketitle

\begin{abstract}
Existence of two solutions to a parametric singular quasi-linear elliptic problem is proved. The equation is driven by the $\Phi$-Laplacian operator and the reaction term can be non-monotone. The main tools employed are a local minimum theorem and the Mountain Pass theorem, together with the truncation technique. Global $C^{1,\tau}$ regularity of solutions is also investigated, chiefly via a priori estimates and perturbation techniques.
\end{abstract}

{
\let\thefootnote\relax
\footnote{\Letter \quad Corresponding author: Umberto Guarnotta (umberto.guarnotta@unipa.it).}
}
\setcounter{footnote}{0}


\section{Introduction and main results}

In this paper we consider the problem
\begin{equation}
\label{lambdaprob}
\tag{${\rm P}_{\lambda,f}$}
\left\{
\begin{alignedat}{2}
-\Delta_\Phi u &= \lambda f(x,u) \quad &&\mbox{in} \;\; \Omega, \\
u &> 0 \quad &&\mbox{in} \;\; \Omega, \\
u &= 0 \quad &&\mbox{on} \;\; \partial \Omega, 
\end{alignedat}
\right.
\end{equation}
where $ \Omega \subseteq \R^N $, $ N \geq 2 $, is a bounded domain with smooth boundary $\partial \Omega$, $\lambda>0$ is a parameter, $ f: \Omega \times (0,+\infty) \to [0,+\infty) $ is a Carathéodory function, and $\Delta_\Phi$ is the $\Phi$-Laplacian, namely,
\begin{equation}
\label{Uhlenbeck}
\Delta_\Phi u := \Div (a(|\nabla u|)\nabla u)
\end{equation}
for a suitable $C^1$ function $a:(0,+\infty) \to (0,+\infty)$. Setting $\phi(t) = ta(t)$ for all $t>0$, we denote by $\Phi$ the primitive of $\phi$ satisfying $\Phi(0)=0$. With the hypotheses below (cf. also \cite[Appendix I]{GG}), $\Phi$ turns out to be the Young function generated by $\phi$ (see \cite[Definition 3.2.1]{KJF}).

\begin{defi}
\label{soldef}
We say that $u \in W^{1,\Phi}_0(\Omega)$ is a (weak) solution to \eqref{lambdaprob} if $u>0$ in $\Omega$ and, for any $v \in W^{1,\Phi}_0(\Omega)$, one has both $f(\cdot,u)v \in L^1(\Omega)$ and
\begin{equation*}
\int_\Omega a(|\nabla u(x)|) \nabla u(x) \cdot \nabla v(x) \dx = \lambda \int_\Omega f(x,u(x)) v(x) \dx.
\end{equation*}
\end{defi}

We assume the following hypotheses (the indices $i_\Psi,s_\Psi$ are defined in \eqref{indices} below):
\begin{itemize}
\item[\underline{${\rm H(a)_1}$}:]
\begin{equation}
\label{ellipt}
-1 < \inf_{t>0} \frac{ta'(t)}{a(t)} \leq \sup_{t>0} \frac{ta'(t)}{a(t)} < +\infty.
\end{equation}
\item[\underline{${\rm H(a)_2}$}:] We suppose that
\begin{equation}
\label{plessN}
\int_1^{+\infty} \Theta_\Phi(t) \dt = +\infty, \quad \mbox{where} \;\; \Theta_\Phi(t) := \frac{\Phi^{-1}(t)}{t^{1+\frac{1}{N}}}.
\end{equation}
Accordingly, the Sobolev-Orlicz conjugate $\Phi_*$ is well defined; see Definition \ref{psistardef}. We also suppose $s_\Phi<i_{\Phi_*}$.
\item[\underline{${\rm H(f)_1}$}:]One has
\begin{equation}
\label{fbelow}
\liminf_{s \to 0^+} f(x,s) = +\infty \quad \mbox{uniformly w.r.t.} \;\; x \in \Omega.
\end{equation}
\item[\underline{${\rm H(f)_2}$}:] There exist $ c_i > 0 $, $ i=1,2 $, $ \gamma \in (0,1) $, and a Young function $\Upsilon$ such that $1<i_\Upsilon \leq s_\Upsilon < i_{\Phi_*}$ and
\begin{equation}
\label{fgrowth}
f(x,s) \leq c_1 \overline{\Upsilon}^{-1}(\Upsilon(s)) + c_2 s^{-\gamma}
\end{equation}
for almost all $x \in \Omega$ and all $s>0$, where $\overline{\Upsilon}$ is the Young conjugate of $\Upsilon$ (see Definition \ref{conj}).
\item[\underline{${\rm H(f)_3}$}:] There exist $R>0$ and $\mu>s_\Phi$ such that
\begin{equation}
\label{AR+}
\mu F(x,t) \leq tf(x,t), \quad \mbox{being} \;\; F(x,t) := \int_R^t f(x,s) \ds,
\end{equation}
for almost all $x\in\Omega$ and all $t \geq R$.
\end{itemize}
\begin{rmk}
Let us briefly comment the main assumptions we have done.
\begin{itemize}
\item Hypothesis ${\rm H(a)_1}$ is called \textit{ellipticity condition} for operators with Uhlenbeck structure, namely, in the form \eqref{Uhlenbeck}. It implies $1<i_\Phi \leq s_\Phi <+\infty$, which in turn implies $ \Phi \in \Delta_2 \cap \nabla_2 $ (see \eqref{deltanabla} below).
\item The first part of ${\rm H(a)_2}$ is the Sobolev-Orlicz analogue of the Sobolev hypothesis $ p < N $. As customary, ${\rm H(a)_2}$ forces the problem in the worst regularity setting, because of the lack of the Morrey-type embedding $ W^{1,\Phi}_0(\Omega) \hookrightarrow C^{0,\tau(\cdot)}(\Omega)$ (see \cite[Theorem 7.4.4]{KJF} for a complete statement). Regarding the second part of ${\rm H(a)_2}$, the condition $s_\Phi<i_{\Phi_*}$ is used only in Section 3.
\item The requirement $s_\Upsilon<i_{\Phi_*}$ in ${\rm H(f)_2}$ is made for the sake of simplicity: indeed, it implies $\Upsilon \ll \Phi_*$ (see \eqref{chainimpl}), which in turn guarantees the compactness of the embedding $W^{1,\Phi}_0(\Omega) \hookrightarrow L^\Upsilon(\Omega)$. Actually, in Section 3 it suffices to require $s_\Upsilon \leq i_{\Phi_*}$. In this respect, see also Remark \ref{noregularity} in the appendix.
\item Condition \eqref{fgrowth} parallels the sub-critical growth condition used in the standard Sobolev setting. We recall that $\overline{\Upsilon}^{-1}(\Upsilon(s))$ can be replaced with $\frac{\Upsilon(s)}{s}$, according to the inequalities
\begin{equation}
\label{equivalent}
\Psi(s) \leq s\overline{\Psi}^{-1}(\Psi(s)) \leq 2\Psi(s) \quad \forall s>0,
\end{equation}
valid for any Young function $\Psi$. For a proof of \eqref{equivalent}, vide \cite[Proposition 2.1.1]{RR}.
\item Hypothesis ${\rm H(f)_3}$ is an adaptation of the \textit{Ambrosetti-Rabinowitz unilateral condition} (see, e.g., \cite[p.154]{PapSmy}) in the Sobolev-Orlicz setting. Following \cite{CGHMS}, ${\rm H(f)_3}$ can be weakened by requiring, instead of $\mu>s_\Phi$,
\begin{equation*}
\mu > \limsup_{t\to+\infty} \frac{t\phi(t)}{\Phi(t)}.
\end{equation*}
\end{itemize}
\end{rmk}

The primary aim of the present work is to extend the results of \cite{CGP,H} to problems driven by non-homogeneous operators as the $(p,q)$-Laplacian $\Delta_p+\Delta_q$, being $\Delta_r u := |\nabla u|^{r-2}\nabla u$ the classical $r$-Laplacian, $r \in (1,+\infty)$. A class of operators encompassing the $(p,q)$-Laplacian is the one described in \cite[Appendix I]{GG}, where ellipticity and Uhlenbeck structure are coupled with a $p$-growth condition that allows to work in the Sobolev setting $W^{1,p}_0(\Omega)$. This class can be extended further, up to the $\Phi$-Laplacian operator, for which regularity theory and maximum principles are still available (see \cite{L,PS}). Existence and regularity results for problems involving the $\Phi$-Laplacian can be found, e.g., in \cite{CGHMS,FIN,TF,CGS}. Dealing with $\Phi$-Laplacian problems requires the usage of Sobolev-Orlicz spaces, since $\Phi$ may have non-standard growth; this fact is discussed in the appendix, where a class of explicit examples is furnished. An introductory exposition about Orlicz and Sobolev-Orlicz spaces is provided \cite[Chapters 3 and 7]{KJF}; we also address to the monographs \cite{KR,RR}. The relation between Sobolev-Orlicz spaces and PDEs is the subject of \cite{G}.

Also singular $\Phi$-Laplacian problems have been studied during the last years. The model case $f(x,u)=a(x)u^{-\gamma}$, with $a \geq 0$ and $\gamma>0$, was investigated in \cite{SGC}. A more general problem, including also convection terms (that is, $f$ depends also on $\nabla u$), was studied in \cite{CGSS}. Due to the lack of variational setting, primarily caused by the strongly singular term (i.e., $\gamma>1$), both works make use of a generalized Galerkin method to get a solution. We are not aware of other existence results pertaining singular $\Phi$-Laplacian problems.

In spite of the papers mentioned above, our approach is variational: first we construct a sub-solution $\underline{u}$ (Lemma \ref{subsollemma}) and truncate $f$ at the level of $\underline{u}$; then we consider the truncated problem \eqref{lambdaauxprob}, which is equivalent to \eqref{lambdaprob} (cf. Lemma \ref{equivalence}), and find a solution by means of the local minimum theorem reported in Theorem \ref{varprinc}. In order to get a second solution, we use the Mountain Pass theorem, jointly with the Ambrosetti-Rabinowitz unilateral condition (see ${\rm H(f)_3}$ above), which implies the Palais-Smale condition (vide Lemma \ref{palaissmale}). We highlight that the Ambrosetti-Rabinowitz condition has been used in the context of $\Phi$-Laplacian problems also in \cite{CGHMS,TF}, while \cite{FIN} uses the Mountain Pass theorem without the Palais-Smale condition. Here we highlight the fact we find the first solution without using the $W^{1,\Phi}$ versus $C^1$ local minimizer technique.

It is worth noticing that the aforementioned works \cite{SGC,CGSS}, that make no use of the Mountain Pass theorem, consider reaction terms with growth not faster than $\Phi$ (usually called `linear'); on the contrary we treat, with the same technique, both linear forcing terms and super-linear ones (but `sub-critical', i.e., growing slower than $\Phi_*$). This is remarkable in our context since, in a variational setting, linear problems possess coercive energy functionals, allowing to find a solution via the Weierstrass-Tonelli theorem instead of the Mountain Pass one. To the best of our knowledge, this is the first work treating singular $\Phi$-Laplacian problems with this technique, that offers a unified approach to the coercive and the non-coercive cases.

Regularity of solutions is investigated in Section 3. $L^\infty$ estimates are provided in Lemma \ref{moser} by using a technique introduced by De Giorgi; see, e.g., \cite[Lemma 2.5.4]{LU}. Then $C^{1,\tau}$ regularity is obtained in Theorem \ref{holder} via the perturbation technique developed by Campanato \cite{Ca1,Ca2}, Giaquinta and Giusti \cite{GiaGiu}, combined with a result pertaining solutions to singular semi-linear elliptic problems that traces back to \cite{GST} (see also \cite{H}).

In the appendix we discuss about the importance of using Sobolev-Orlicz spaces, providing also two examples of problems fulfilling ${\rm H(a)_1}$--${\rm H(a)_2}$ and ${\rm H(f)_1}$--${\rm H(f)_3}$.

\section{Preliminaries}

We denote by $d(x)$ the distance of $x \in \Omega$ from $\partial \Omega$, while $d_\Omega$ stands for the diameter of $\Omega$. Given any function $u: \Omega \to \R$ and any number $\rho \in \R$, $\{u<\rho\}$ stands for the set $\{x \in \Omega: u(x)<\rho\}$, and the same holds for $\{u \geq \rho\}$, $\{u = \rho\}$, etc. \\
To avoid unnecessary technicalities, hereafter we use `for all $x\in\Omega$' instead of `for almost all $x\in\Omega$' when no confusion arises.

\begin{defi}
A function $\Psi:[0,+\infty) \to [0,+\infty)$ is said to be a Young function\footnote{Some textbooks, as \cite{AF}, use the notion of N-function; here we adopt the nomenclature used in \cite{KJF}. See \cite[Section 8.1]{AF} and \cite[Remark 3.2.7]{KJF} for further details.} if it is continuous, strictly increasing, convex, and the following holds true:
\begin{equation}
\label{Nfunct}
\lim_{t \to 0^+} \frac{\Psi(t)}{t} = 0, \quad \lim_{t \to +\infty} \frac{\Psi(t)}{t} = +\infty.
\end{equation}
\end{defi}

\begin{defi}
\label{conj}
Let $\Psi$ be a Young function. We denote by $\overline{\Psi}$ the Young conjugate of $\Psi$, defined via Legendre transformation as
\begin{equation*}
\overline{\Psi}(t) := \max_{s \geq 0} \{st-\Psi(s)\} \quad \forall t \geq 0.
\end{equation*}
\end{defi}

\begin{defi}
\label{psistardef}
Let $\Psi$ be a Young function satisfying \eqref{plessN} with $\Psi$ in place of $\Phi$. Suppose also, without loss of generality (cf. \cite[Exercise 7.2.2]{KJF}), that
\begin{equation*}
\int_0^1 \Theta_{\Psi}(s) \ds < +\infty.
\end{equation*}
The Sobolev-Orlicz conjugate of $\Psi$, indicated as $\Psi_*$, is defined via its inverse as
\begin{equation*}
\Psi_*^{-1}(t) := \int_0^t \Theta_\Psi(s) \ds.
\end{equation*}
\end{defi}

\begin{defi}
Let $\Psi$ be a Young function. We write $\Psi \in \Delta_2$ if there exist $k,T>0$ such that
\begin{equation*}
\Psi(2t) \leq k\Psi(t) \quad \forall t \geq T.
\end{equation*}
%
%
We write $\Psi \in \nabla_2$ if there exist $\eta>1$ and $T>0$ such that
\begin{equation*}
\Psi(t) \leq \frac{1}{2\eta}\Psi(\eta t) \quad \forall t \geq T.
\end{equation*}
\end{defi}

Equivalent statements are collected in \cite[Theorem 2.3.3 and Corollary 2.3.4]{RR}; here we only mention
\begin{equation}
\label{deltanabla}
\begin{split}
\Psi \in \Delta_2 \quad &\Leftrightarrow \quad \overline{\Psi} \in \nabla_2 \quad \Leftrightarrow \quad \limsup_{t\to+\infty} \frac{t\Psi'(t)}{\Psi(t)} < +\infty, \\
\Psi \in \nabla_2 \quad &\Leftrightarrow \quad \overline{\Psi} \in \Delta_2 \quad \Leftrightarrow \quad \liminf_{t\to+\infty} \frac{t\Psi'(t)}{\Psi(t)} > 1.
\end{split}
\end{equation}
For a comparison with power-law functions, see \cite[Corollary 2.3.5]{RR}.

Let $\Psi \in \Delta_2$. We endow the Orlicz space\footnote{Since $\Psi \in \Delta_2$, we make no distinction between \textit{Orlicz space} and \textit{Orlicz class}; see \cite[Theorem 3.7.3]{KJF}.}
\begin{equation*}
L^\Psi(\Omega) := \left\{ u:\Omega \to \R \; \mbox{measurable}: \, \int_\Omega \Psi(|u(x)|) \dx < +\infty \right\}
\end{equation*}
with the Luxembourg norm
\begin{equation*}
\|u\|_{L^\Psi(\Omega)} := \inf \left\{ \lambda>0: \, \int_\Omega \Psi\left(\frac{|u(x)|}{\lambda}\right) \dx \leq 1 \right\}.
\end{equation*}

Suppose that
\begin{equation}
\label{indices}
1 < i_\Psi := \inf_{t>0} \frac{t\Psi'(t)}{\Psi(t)} \leq \sup_{t>0} \frac{t\Psi'(t)}{\Psi(t)} =: s_\Psi < +\infty,
\end{equation}
which implies $\Psi \in \Delta_2 \cap \nabla_2$ by \eqref{deltanabla}. We define the functions $\underline{\zeta}_\Psi,\overline{\zeta}_\Psi:[0,+\infty) \to [0,+\infty)$ as
\begin{equation*}
\underline{\zeta}_\Psi(t) := \min\{t^{i_\Psi},t^{s_\Psi}\}, \quad \overline{\zeta}_\Psi(t) := \max\{t^{i_\Psi},t^{s_\Psi}\}.
\end{equation*}
One has (cf. \cite[Lemma 2.1]{FIN})
\begin{equation}
\label{factor}
\underline{\zeta}_\Psi(k) \Psi(t) \leq \Psi(kt) \leq \overline{\zeta}_\Psi(k) \Psi(t) \quad \forall k,t \geq 0
\end{equation}
and
\begin{equation}
\label{comparison}
\begin{split}
&\underline{\zeta}_\Psi(\|w\|_{L^\Psi(\Omega)}) \leq \int_\Omega \Psi(|w(x)|) \dx \leq \overline{\zeta}_\Psi(\|w\|_{L^\Psi(\Omega)})
\end{split}
\end{equation}
for all $w \in L^\Psi(\Omega)$. We also recall (see \cite[Lemmas 2.4-2.5]{FIN}) that
\begin{equation}
\label{youngind}
s_\Psi' \leq i_{\overline{\Psi}} \leq s_{\overline{\Psi}} \leq i_\Psi'
\end{equation}
and, provided $s_\Psi<N$,
\begin{equation}
\label{sobind}
i_\Psi^* \leq i_{\Psi_*} \leq s_{\Psi_*} \leq s_\Psi^*,
\end{equation}
being $r' := \frac{r}{r-1}$ and $r^* := \frac{Nr}{N-r}$ respectively the Young and the Sobolev conjugates of $r$.

\begin{defi}
Let $\Psi_1,\Psi_2$ be two Young functions. We write $\Psi_1 < \Psi_2$ if there exist $c,T>0$ such that
\begin{equation}
\label{Youngcomp}
\Psi_1(t) \leq \Psi_2(ct) \quad \forall t \geq T.
\end{equation}
We write $\Psi_1 \ll \Psi_2$ if for any $c>0$ there exists $T=T(c)>0$ such that \eqref{Youngcomp} holds true. Equivalently,
\begin{equation*}
\lim_{t \to +\infty} \frac{\Psi_1(t)}{\Psi_2(\eta t)} = 0 \quad \forall \eta > 0.
\end{equation*}
\end{defi}

Further characterizations can be found in \cite[Theorem 2.2.2]{RR}. It is worth recalling the following chain of (non-reversible) implications:
\begin{equation}
\label{chainimpl}
s_{\Psi_1} < i_{\Psi_2} \quad \Rightarrow \quad \Psi_1 \ll \Psi_2 \quad \Rightarrow \quad \Psi_1 < \Psi_2.
\end{equation}

We consider the Sobolev-Orlicz space
\begin{equation*}
W^{1,\Phi}(\Omega) := \{ u \in L^\Phi(\Omega): \, |\nabla u| \in L^\Phi(\Omega)\},
\end{equation*}
equipped with the norm $\|u\|_{W^{1,\Phi}(\Omega)} := \|u\|_{L^\Phi(\Omega)} + \|\nabla u\|_{L^\Phi(\Omega)}$, and its subspace $W^{1,\Phi}_0(\Omega)$, which is the closure of $C^\infty_c(\Omega)$ under $\|\cdot\|_{W^{1,\Phi}(\Omega)}$. According to the Poincaré inequality (see, e.g., \cite[p.8]{CGSS}), we are allowed to endow $W^{1,\Phi}_0(\Omega)$ with the norm
\begin{equation*}
\|u\|_{W^{1,\Phi}_0(\Omega)} := \|\nabla u\|_{L^\Phi(\Omega)}.
\end{equation*}
Since $\Phi \in \Delta_2 \cap \nabla_2$, the space $W^{1,\Phi}_0(\Omega)$ is separable and reflexive (cf. \cite[Theorem 8.31]{AF}). Its dual space will be denoted by $W^{-1,\overline{\Phi}}(\Omega)$, while $\langle\cdot,\cdot\rangle$ represent the duality brackets between $W^{-1,\overline{\Phi}}(\Omega)$ and $W^{1,\Phi}_0(\Omega)$. We recall that $ W^{1,\Phi}_0(\Omega) \hookrightarrow L^{\Phi_*}(\Omega)$ continuously and $W^{1,\Phi}_0(\Omega) \hookrightarrow L^\Upsilon(\Omega)$ compactly for all $\Upsilon \ll \Phi_*$; see \cite[Theorems 7.2.3 and 7.4.4]{KJF}.

Although the next result is folklore, we briefly sketch its proof for the sake of completeness.

\begin{lemma}
\label{princprop}
Under ${\rm H(a)_1}$, the operator $ -\Delta_\Phi: W^{1,\Phi}_0(\Omega) \to W^{-1,\overline{\Phi}}(\Omega) $ defined as
\begin{equation*}
\langle -\Delta_\Phi u,v \rangle := \int_\Omega a(|\nabla u|)\nabla u \cdot \nabla v \dx \quad \forall u,v \in W^{1,\Phi}_0(\Omega)
\end{equation*}
is well defined, bounded, continuous, coercive, strictly monotone, and of type ${\rm (S_+)}$. Moreover, the functional $ H: W^{1,\Phi}_0(\Omega) \to \R $ defined as
\begin{equation}
\label{Hdef}
H(u) := \int_\Omega \Phi(|\nabla u|) \dx
\end{equation}
is convex, weakly lower semi-continuous, and of class $ C^1 $, with $ H' = -\Delta_\Phi $ in $ W^{-1,\overline{\Phi}}(\Omega) $.
\end{lemma}
\begin{proof}
According to the H\"older inequality (see \cite[p.62]{G}) we get
\begin{equation}
\label{welldef}
\begin{split}
|\langle -\Delta_\Phi u,v \rangle| &\leq \int_\Omega \phi(|\nabla u|) |\nabla v| \dx \leq s_\Phi \int_\Omega \frac{\Phi(|\nabla u|)}{|\nabla u|} |\nabla v| \dx \\
&\leq s_\Phi \left\| \frac{\Phi(|\nabla u|)}{|\nabla u|} \right\|_{L^{\overline{\Phi}}(\Omega)} \|\nabla v\|_{L^\Phi(\Omega)} = s_\Phi \left\| \frac{\Phi(|\nabla u|)}{|\nabla u|} \right\|_{L^{\overline{\Phi}}(\Omega)} \|v\|_{W^{1,\Phi}_0(\Omega)}.
\end{split}
\end{equation}
Exploiting \eqref{equivalent} we infer
\begin{equation}
\label{welldef2}
\int_\Omega \overline{\Phi}\left( \frac{\Phi(|\nabla u|)}{|\nabla u|} \right) \dx \leq \int_\Omega \Phi(|\nabla u|) \dx < +\infty.
\end{equation}
By \eqref{welldef}--\eqref{welldef2} we deduce that $-\Delta_\Phi$ is well defined. Boundedness and continuity follow from \eqref{welldef2} and \cite[Lemma 7.3]{SGC}.

In order to prove the coercivity of $-\Delta_\Phi$, we exploit \eqref{comparison} to obtain
\begin{equation}
\label{coerc}
\begin{split}
\int_\Omega a(|\nabla u|) |\nabla u|^2 \dx &= \int_\Omega \phi(|\nabla u|) |\nabla u| \dx \geq i_\Phi \int_\Omega \Phi(|\nabla u|) \dx \\
&\geq i_\Phi \underline{\zeta}_\Phi(\|\nabla u\|_{L^\Phi(\Omega)})
\end{split}
\end{equation}
for all $u \in W^{1,\Phi}_0(\Omega)$. Hence
\begin{equation*}
\frac{\langle -\Delta_\Phi u,u \rangle}{\|u\|_{W^{1,\Phi}_0(\Omega)}} \geq i_\Phi \frac{\underline{\zeta}_\Phi(\|u\|_{W^{1,\Phi}_0(\Omega)})}{\|u\|_{W^{1,\Phi}_0(\Omega)}} \to +\infty \quad \mbox{as} \;\; \|u\|_{W^{1,\Phi}_0(\Omega)} \to +\infty.
\end{equation*}

The strict monotonicity and the ${\rm (S_+)}$ property of $-\Delta_\Phi$ are guaranteed by \cite[Propositions A.2-A.3]{CGS}.

Convexity of $ H $ directly follows from convexity of $ \Phi $, while Lebesgue's dominated convergence theorem and \cite[Lemma 7.3]{SGC} ensure that $ H $ is continuous. As a consequence, $ H $ is weakly lower semi-continuous. The fact that $ H $ is of class $ C^1 $ has been proved in \cite[Lemma A.3]{FIN}.
\end{proof}

Firstly we consider the problem
\begin{equation}
\label{prob}
\tag{${\rm P}_{1,f}$}
\left\{
\begin{alignedat}{2}
-\Delta_\Phi u &= f(x,u) \quad &&\mbox{in} \;\; \Omega, \\
u &> 0 \quad &&\mbox{in} \;\; \Omega, \\
u &= 0 \quad &&\mbox{on} \;\; \partial \Omega.
\end{alignedat}
\right.
\end{equation}

\begin{lemma}
\label{subsollemma}
Suppose ${\rm H(a)_1}$ and ${\rm H(f)_1}$. Then problem \eqref{prob} admits a sub-solution $ \underline{u} \in C^{1,\tau}_0(\overline{\Omega}) $, with $ \tau \in (0,1] $ opportune, satisfying
\begin{equation}
\label{subsolprops}
k_1 d(x) \leq \underline{u}(x) \leq k_2 d(x) \quad \forall x \in \Omega
\end{equation}
for suitable $ k_1,k_2 > 0 $.
\end{lemma}
\begin{proof}
This proof is patterned after the one of \cite[Lemma 3.5]{GMM}. Hypothesis ${\rm H(f)_1}$ provides $ \delta> 0 $ such that
\begin{equation}
\label{fbelow2}
f(x,s) \geq 1 \quad \mbox{for all} \;\; (x,s) \in \Omega \times (0,\delta).
\end{equation}
For any $n \in \N$, let us consider the following problem:
\begin{equation}
\label{torsion}
\tag{${\rm P}_{1,\frac{1}{n}}$}
\left\{
\begin{alignedat}{2}
-\Delta_\Phi u &= \frac{1}{n} \quad &&\mbox{in} \;\; \Omega, \\
u &= 0 \quad &&\mbox{on} \;\; \partial \Omega.
\end{alignedat}
\right.
\end{equation}
By virtue of Lemma \ref{princprop}, Minty-Browder's theorem \cite[Theorem 5.16]{B} can be applied; thus \eqref{torsion} admits a unique solution $ u_n \in W^{1,\Phi}_0(\Omega) $. Lieberman's nonlinear regularity theory \cite[Theorem 1.7]{L} guarantees that $ \{u_n: \, n \in \N \} $ is bounded in $ C^{1,\tau}_0(\overline{\Omega}) $ for some $ \tau \in (0,1] $. Hence, thanks to the Ascoli-Arzelà theorem and up to subsequences, we get $ u_n \to u $ in $ C^1_0(\overline{\Omega}) $ for some $ u \in C^1_0(\overline{\Omega}) $. Passing to the limit in the weak formulation of \eqref{torsion} reveals that $ u \equiv 0 $ in $ \Omega $. Hence it is possible to choose $ \hat{n} \in \N $ such that
\begin{equation}
\label{subsolbound}
\|u_{\hat{n}}\|_{L^\infty(\Omega)} < \delta.
\end{equation}
Set $ \underline{u} := u_{\hat{n}} $. The strong maximum principle \cite[Theorem 1.1.1]{PS} ensures $ \underline{u} > 0 $ in $ \Omega $. Thus, by \eqref{fbelow2}--\eqref{subsolbound} one has
\begin{equation*}
-\Delta_\Phi \underline{u} = \frac{1}{\hat{n}} \leq 1 \leq f(x,\underline{u})
\end{equation*}
in weak sense. A standard argument involving the Boundary Point lemma \cite[Theorem 5.5.1]{PS} and the H\"older continuity of $ \nabla \underline{u} $ gives \eqref{subsolprops}.
\end{proof}

\begin{rmk}
\label{deltaR}
Without loss of generality, one can choose $\delta < R$ in \eqref{fbelow2}, where $R>0$ comes from ${\rm H(f)_3}$. Hereafter we make this assumption, which yields $ \|\underline{u}\|_{L^\infty(\Omega)} \leq \delta < R$, according to \eqref{subsolbound}.
\end{rmk}

Let us consider the auxiliary problem
\begin{equation}
\label{auxprob}
\tag{${\rm P}_{1,\hat{f}}$}
\left\{
\begin{alignedat}{2}
-\Delta_\Phi u &= \hat{f}(x,u) \quad &&\mbox{in} \;\; \Omega, \\
u &= 0 \quad &&\mbox{on} \;\; \partial \Omega, 
\end{alignedat}
\right.
\end{equation}
where $ \hat{f}: \Omega \times \R \to [0,+\infty) $ is defined as
\begin{equation*}
\hat{f}(x,s) := \left\{
\begin{alignedat}{1}
f(x,\underline{u}(x)) \quad &\mbox{if} \;\; |s| \leq \underline{u}(x), \\
f(x,|s|) \quad &\mbox{if} \;\; |s| > \underline{u}(x),
\end{alignedat}
\right.
\end{equation*}
being $ \underline{u} $ as in Lemma \ref{subsollemma}. We also define
\begin{equation*}
\hat{F}(x,s) := \int_0^s \hat{f}(x,t) \dt.
\end{equation*}
Exploiting \eqref{fgrowth} and \eqref{subsolprops} one has
\begin{equation}
\label{fest}
\begin{split}
0\leq \hat{f}(x,s) &\leq c_1 \overline{\Upsilon}^{-1}(\Upsilon(|s|)) + c_1 \overline{\Upsilon}^{-1}(\Upsilon(\underline{u}(x))) + c_2 \underline{u}(x)^{-\gamma} \\
&\leq c_1 \overline{\Upsilon}^{-1}(\Upsilon(|s|)) + \alpha d(x)^{-\gamma} + \beta,
\end{split}
\end{equation}
being $ \alpha,\beta>0 $ such that
\begin{equation*}
\alpha := c_2 k_1^{-\gamma}, \quad \beta:= c_1 \overline{\Upsilon}^{-1}(\Upsilon(k_2 d_\Omega)).
\end{equation*}

\begin{lemma}
\label{equivalence}
Let ${\rm H(a)_1}$ and ${\rm H(f)_1}$ be satisfied. Then any $ u \in W^{1,\Phi}_0(\Omega) $ weak solution to \eqref{auxprob} is a weak solution to \eqref{prob} and vice-versa.
\end{lemma}

\begin{proof}
To show the equivalence of \eqref{auxprob} and \eqref{prob}, it suffices to prove that any solution to either \eqref{auxprob} or \eqref{prob} is greater than $\underline{u}$; then the conclusion will follow by the definition of $ \hat{f} $.

Let $u \in W^{1,\Phi}_0(\Omega)$ be a weak solution to \eqref{auxprob}. The weak maximum principle, jointly with $ \hat{f} \geq 0 $, ensures $ u \geq 0 $. Then Lemma \ref{subsollemma} and the weak comparison principle (cf., e.g., \cite[Theorem 3.4.1]{PS}), applied on $ u $ and $ \underline{u} $, yields $ u \geq \underline{u} $: indeed $-\Delta_\Phi$ is a strictly monotone operator (see Lemma \ref{princprop}) and, in weak sense,
\begin{equation*}
-\Delta_\Phi \underline{u} \leq f(x,\underline{u}) = \hat{f}(x,u) = -\Delta_\Phi u \quad \mbox{in} \;\; \{x \in \Omega: \, u(x) \leq \underline{u}(x)\}.
\end{equation*}

Now let $u \in W^{1,\Phi}_0(\Omega)$ be a weak solution to \eqref{prob}. Recalling that $\|\underline{u}\|_{L^\infty(\Omega)} \leq \delta$ by \eqref{subsolbound}, from \eqref{fbelow2} we get
\begin{equation*}
-\Delta_\Phi \underline{u} = \frac{1}{\hat{n}} \leq 1 \leq f(x,u) = -\Delta_\Phi u \quad \mbox{in} \;\; \{x \in \Omega: \, u(x) \leq \underline{u}(x)\},
\end{equation*}
where $\delta,\hat{n}$ come from Lemma \ref{subsollemma}. As above, the weak comparison principle ensures $u\geq \underline{u}$.
\end{proof}

\begin{lemma}
\label{reactprop}
Suppose ${\rm H(a)_1}$--${\rm H(a)_2}$ and ${\rm H(f)_1}$--${\rm H(f)_2}$. Then the functional $ K:W^{1,\Phi}_0(\Omega) \to \R $ defined as
\begin{equation*}
K(u) := \int_\Omega \hat{F}(x,u) \dx
\end{equation*}
is well defined, weakly sequentially continuous, and of class $ C^1 $, with
\begin{equation}
\label{Kder}
\langle K'(u),v \rangle = \int_\Omega \hat{f}(x,u)v \dx  \quad \forall u,v \in W^{1,\Phi}_0(\Omega).
\end{equation}
Moreover, $ K':W^{1,\Phi}_0(\Omega) \to W^{-1,\overline{\Phi}}(\Omega) $ is a completely continuous operator.
\end{lemma}

\begin{proof}
Using \eqref{fgrowth} and \eqref{equivalent} we estimate $ \hat{F} $ as
\begin{equation*}
\begin{split}
|\hat{F}(x,s)| &\leq \int_0^{|s|} \hat{f}(x,t) \dt = \int_0^{\underline{u}(x)} f(x,\underline{u}(x)) \dt + \int_{\underline{u}(x)}^{|s|} f(x,t) \dt \\
&\leq \underline{u}(x) \left[ c_1 \overline{\Upsilon}^{-1}(\Upsilon(\underline{u}(x))) + c_2 \underline{u}(x)^{-\gamma} \right] + \int_0^{|s|} \left[ c_1 \overline{\Upsilon}^{-1}(\Upsilon(t)) + c_2 t^{-\gamma} \right] \dt \\
&\leq c_1 \underline{u}(x)\overline{\Upsilon}^{-1}(\Upsilon(\underline{u}(x))) + c_2 \underline{u}(x)^{1-\gamma} + c_1 |s| \overline{\Upsilon}^{-1}(\Upsilon(|s|)) + \frac{c_2}{1-\gamma} |s|^{1-\gamma} \\
&\leq 2c_1 \Upsilon(\underline{u}(x)) + c_2 \underline{u}(x)^{1-\gamma} + 2c_1 \Upsilon(|s|) + \frac{c_2}{1-\gamma} |s|^{1-\gamma}
\end{split}
\end{equation*}
for all $ (x,s) \in \Omega \times \R $. Exploiting \eqref{subsolprops} and $i_\Upsilon > 1$ we obtain
\begin{equation}
\label{Fest}
|\hat{F}(x,s)| \leq C_1 + C_2 \Upsilon(|s|),
\end{equation}
with positive constants
\begin{equation}
\label{constants}
\begin{split}
C_1 &:= 2c_1 \Upsilon(k_2 d_\Omega) + c_2 (k_2 d_\Omega)^{1-\gamma} + \frac{c_2}{1-\gamma}, \\
C_2 &:= 2c_1 + \frac{c_2}{1-\gamma} \frac{1}{\Upsilon(1)}.
\end{split}
\end{equation}
Taking into account also that $W^{1,\Phi}_0(\Omega) \hookrightarrow L^\Upsilon(\Omega)$ because of $\Upsilon \ll \Phi_*$, we deduce that $ K $ is well defined.

Now we compute the G\^{a}teaux derivative of $K$. We fix $v \in W^{1,\Phi}_0(\Omega)$ and apply Torricelli's theorem to deduce
\begin{equation}
\label{Kder2}
\begin{split}
\lim_{t \to 0^+} \frac{K(u+tv)-K(u)}{t} &= \lim_{t \to 0^+} \int_\Omega \frac{\hat{F}(x,u+tv)-\hat{F}(x,u)}{t} \dx \\
&= \lim_{t \to 0^+} \int_\Omega v \left( \int_0^1 \hat{f}(x,u+stv) \ds \right) \dx.
\end{split}
\end{equation}
According to \eqref{fest}, \eqref{equivalent}, the embedding $W^{1,\Phi}_0(\Omega) \hookrightarrow L^\Upsilon(\Omega)$, and the Hardy inequality \cite[Corollary 1]{C}\footnote{We use Hardy's inequality in the form
\begin{equation*}
\|d^{-1}v\|_{L^\Phi(\Omega)} \leq c \|\nabla v\|_{L^\Phi(\Omega)} \quad \forall v \in W^{1,\Phi}_0(\Omega),
\end{equation*}
being $c>0$ opportune. This inequality is valid since $\Phi \in \nabla_2$.}, besides supposing $t\in(0,1)$, we infer
\begin{equation*}
\begin{split}
|v\hat{f}(x,u+stv)| &\leq c_1|v|\overline{\Upsilon}^{-1}(\Upsilon(|u|+|v|)) + \alpha d^{-\gamma}|v| + \beta|v| \\
&\leq c_1(|u|+|v|)\overline{\Upsilon}^{-1}(\Upsilon(|u|+|v|)) + \alpha d^{-\gamma}|v| + \beta|v| \\
&\leq 2c_1\Upsilon(|u|+|v|) + \alpha d_\Omega^{1-\gamma} d^{-1}|v| + \beta|v| \in L^1(\Omega).
\end{split}
\end{equation*}
Hence, we can pass the limit under the integral sign in \eqref{Kder2} and get \eqref{Kder}. The remaining part of the proof follows exactly as in \cite[Lemma 2.3]{CGP}, using the compactness of the embedding $W^{1,\Phi}_0(\Omega) \hookrightarrow L^\Upsilon(\Omega)$, as well as \eqref{fest} and \cite[Lemma 7.3]{SGC}.
\end{proof}

\begin{rmk}
Fix any $u \in W^{1,\Phi}_0(\Omega)$. By \eqref{Fest} we deduce that
\begin{equation}
\label{compact}
\int_{\Omega \cap \{|u| \leq \rho\}} |\hat{F}(x,u)| \dx \leq (C_1 + C_2\Upsilon(\rho))|\Omega| =: \Pi(\rho) \quad \forall \rho \geq 0.
\end{equation}
\end{rmk}

\section{Regularity of solutions}

In this section we prove $C^{1,\alpha}$ regularity up to the boundary for solutions to \eqref{prob}.

Given a measurable function $u:\Omega \to \R$, for any $k \in \R$ we set
\begin{equation*}
\Omega_k := \{x \in \Omega: \, u(x) \geq k\}.
\end{equation*}
\begin{lemma}
Let $p,r>1$. Suppose that $u \in L^p(\Omega)$ satisfies
\begin{equation}
\label{degiorgi}
\left( \int_{\Omega_k} (u-k)^{p} \dx \right)^{\frac{1}{r}} \leq c \left[ \int_{\Omega_k} (u-k)^{p} \dx + k^{p}|\Omega_k| \right] \quad \mbox{for all} \;\; k \geq K,
\end{equation}
for suitable $c,K>0$. Then there exists $M>0$ such that $u \leq M$ in $\Omega$.
\end{lemma}

\begin{proof}
Let us fix $M>2K$ to be chosen later, and set
\begin{equation}
\label{kn}
k_n := M \left( 1-\frac{1}{2^{n+1}} \right) \quad \forall n \in \N_0 := \N \cup \{0\}.
\end{equation}
By \eqref{degiorgi} and \eqref{kn} we have, for all $n \in \N_0$,
\begin{equation}
\label{degiorgin}
\left( \int_{\Omega_{k_{n+1}}} (u-k_{n+1})^{p} \dx \right)^{\frac{1}{r}} \leq c \left[ \int_{\Omega_{k_n}} (u-k_n)^{p} \dx + k_{n+1}^{p}|\Omega_{k_{n+1}}| \right].
\end{equation}
Chebichev's inequality entails
\begin{equation*}
(k_{n+1}-k_n)^{p}|\Omega_{k_{n+1}}| \leq \int_{\Omega_{k_n}} (u-k_n)^{p} \dx.
\end{equation*}
Thus, recalling \eqref{kn} and $k>1$,
\begin{equation}
\label{chebichev}
k_{n+1}^{p}|\Omega_{k_{n+1}}| \leq \frac{k_{n+1}^{p}}{(k_{n+1}-k_n)^{p}} \int_{\Omega_{k_n}} (u-k_n)^{p} \dx \leq 2^{(n+2)p} \int_{\Omega_{k_n}} (u-k_n)^{p} \dx.
\end{equation}
Inserting \eqref{chebichev} into \eqref{degiorgin} gives
\begin{equation}
\label{degiorgin2}
\begin{split}
\int_{\Omega_{k_{n+1}}} (u-k_{n+1})^{p} \dx &\leq \left[ c(2^{(n+2)p}+1) \int_{\Omega_{k_n}} (u-k_n)^{p} \dx \right]^r \\
&\leq Cb^n \left( \int_{\Omega_{k_n}} (u-k_n)^{p} \dx \right)^{1+a}
\end{split}
\end{equation}
for all $n \in \N_0$, where $a:=r-1>0$, $b:=2^{pr}>1$, and $C=C(c,p,r)>0$ is a suitable constant independent of $k$ and $n$. Now we apply the fast geometric convergence lemma \cite[Lemma 2.4.7]{LU} to the sequence $y_n:=\int_{\Omega_{k_n}} (u-k_n)^{p} \dx$, which ensures $y_n \to 0$ provided
\begin{equation}
\label{smallness}
y_0 \leq C^{-\frac{1}{a}}b^{-\frac{1}{a^2}}.
\end{equation}
We can choose $M$, independent of $n$, such that \eqref{smallness} holds true: indeed, by \eqref{kn} and the dominated convergence theorem,
\begin{equation}
\label{small}
y_0 = \int_{\Omega_{k_0}} (u-k_0)^{p} \dx = \int_\Omega \left( u-\frac{M}{2} \right)_+^{p} \dx \xrightarrow{M \to +\infty} 0.
\end{equation}
Keeping $M$ fixed as in \eqref{smallness}--\eqref{small}, from $y_n \to 0$ we obtain
\begin{equation*}
\int_\Omega (u-M)_+^{p} \dx \leq \int_\Omega (u-k_n)_+^{p} \dx = \int_{\Omega_{k_n}} (u-k_n)^{p} \dx \xrightarrow{n \to \infty} 0,
\end{equation*}
which implies $\int_\Omega (u-M)_+^p \dx = 0$, whence $u \leq M$ in $\Omega$.
\end{proof}

\begin{lemma}
\label{moser}
Let ${\rm H(a)_1}$--${\rm H(a)_2}$ and ${\rm H(f)_2}$ be satisfied. Then any $ u \in W^{1,\Phi}_0(\Omega) $ weak solution to \eqref{prob} is essentially bounded in $\Omega$.
\end{lemma}
\begin{proof}
Pick any $k>1$. Testing \eqref{prob} with $(u-k)_+$ and using \eqref{fgrowth} yield
\begin{equation}
\label{mosereq}
\begin{split}
\int_{\Omega_k} \Phi(|\nabla u|) \dx &\leq i_\Phi^{-1} \int_{\Omega_k} \phi(|\nabla u|)|\nabla u| \dx = i_\Phi^{-1} \int_{\Omega_k} f(x,u)(u-k) \dx \\
&\leq i_\Phi^{-1} \left[ c_1 \int_{\Omega_k} \overline{\Upsilon}^{-1}(\Upsilon(u))u \dx + c_2 \int_{\Omega_k} u^{1-\gamma} \dx \right].
\end{split}
\end{equation}
First we estimate each term on the right-hand side of \eqref{mosereq}. By convexity of $\Upsilon$ and \eqref{equivalent} we have, for any $k$ sufficiently large,
\begin{equation}
\label{singular}
u^{1-\gamma} \leq \Upsilon(u) \leq \overline{\Upsilon}^{-1}(\Upsilon(u))u \quad \mbox{in} \;\; \Omega_k.
\end{equation}
Moreover, by \eqref{equivalent} and \eqref{factor}, besides $s_\Upsilon \leq i_{\Phi_*}$ and $k>1$, it turns out that
\begin{equation}
\label{regular}
\begin{split}
\int_{\Omega_k} \overline{\Upsilon}^{-1}(\Upsilon(u))u \dx &\leq 2 \int_{\Omega_k} \Upsilon(u) \dx \leq 2 \Upsilon(1) \int_{\Omega_k} u^{s_\Upsilon} \dx \leq 2 \Upsilon(1) \int_{\Omega_k} u^{i_{\Phi_*}} \dx \\
&\leq 2^{i_{\Phi_*}} \Upsilon(1) \left[ \int_{\Omega_k} (u-k)^{i_{\Phi_*}} \dx + k^{i_{\Phi_*}}|\Omega_k| \right].
\end{split}
\end{equation}
On the other hand, to estimate the left-hand side of \eqref{mosereq}, we observe that the Sobolev embedding theorem and $ t^{i_{\Phi_*}}<\Phi_*$ in the sense of \eqref{Youngcomp} (see \eqref{factor}) yield
\begin{equation*}
W^{1,\Phi}_0(\Omega) \hookrightarrow L^{\Phi_*}(\Omega) \hookrightarrow L^{i_{\Phi_*}}(\Omega),
\end{equation*}
where the latter space is a Lebesgue space. We deduce the embedding inequality
\begin{equation*}
c\|w\|_{L^{i_{\Phi_*}}(\Omega)} \leq \|\nabla w\|_{L^\Phi(\Omega)} \quad \forall w \in W^{1,\Phi}_0(\Omega)
\end{equation*}
for a suitable $c>0$. Thus, choosing $w=(u-k)_+$, from \eqref{comparison} we get
\begin{equation*}
\int_{\Omega_k} \Phi(|\nabla u|) \dx \geq \underline{\zeta}_{\Phi}(\|\nabla u\|_{L^\Phi(\Omega)}) \geq \underline{\zeta}_{\Phi}(\|\nabla (u-k)_+\|_{L^\Phi(\Omega)}) \geq \underline{\zeta}_{\Phi}(c\|(u-k)_+\|_{L^{i_{\Phi_*}}(\Omega)}).
\end{equation*}
Reasoning as in \eqref{small}, for any $k$ big enough we get $c\|(u-k)_+\|_{L^{i_{\Phi_*}}(\Omega)}\leq 1$, so that
\begin{equation}
\label{embedding}
\int_{\Omega_k} \Phi(|\nabla u|) \dx \geq c^{s_\Phi} \left( \int_{\Omega_k} (u-k)^{i_{\Phi_*}} \dx \right)^{\frac{s_\Phi}{i_{\Phi_*}}}.
\end{equation}
Inserting \eqref{singular}--\eqref{embedding} into \eqref{mosereq} we deduce
\begin{equation*}
\left( \int_{\Omega_k} (u-k)^{i_{\Phi_*}} \dx \right)^{\frac{s_\Phi}{i_{\Phi_*}}} \leq C \left[ \int_{\Omega_k} (u-k)^{i_{\Phi_*}} \dx + k^{i_{\Phi_*}}|\Omega_k| \right]
\end{equation*}
for a sufficiently large $C>0$. Hence applying Lemma \ref{degiorgi} with $p=i_{\Phi_*}>1$ and $r=\frac{i_{\Phi_*}}{s_\Phi}>1$ (cf. ${\rm H(a)_2}$) yields the conclusion.
\end{proof}

\begin{rmk}
The $L^\infty$ estimate provided in Lemma \ref{moser} is valid also when $s_\Upsilon = i_{\Phi_*}$ which, in the classical Sobolev setting, represents the \textit{critical case}; hence $M$ must depend on the solution $u$. On the other hand, in the \textit{sub-critical case} $s_\Upsilon<i_{\Phi_*}$ this estimate can be improved, and it turns out that $M$ depends only on $\|u\|_{W^{1,\Phi}_0(\Omega)}$ instead of $u$ itself.
\end{rmk}

\begin{thm}
\label{holder}
Let ${\rm H(a)_1}$--${\rm H(a)_2}$ and ${\rm H(f)_1}$--${\rm H(f)_2}$ be satisfied. Then any $ u \in W^{1,\Phi}_0(\Omega) $ weak solution of \eqref{prob} belongs to $C^{1,\tau}_0(\overline{\Omega})$ for some $\tau \in (0,1]$.
\end{thm}

\begin{proof}
Lemma \ref{moser} guarantees that $ u \in L^\infty(\Omega) $. In addition, Lemma \ref{equivalence} ensures that $u$ solves also \eqref{auxprob}. Using $u\in L^\infty(\Omega)$ and \eqref{fest} we get
\begin{equation*}
0 \leq \hat{f}(x,u(x)) \leq C d(x)^{-\gamma} \quad \forall x \in \Omega,
\end{equation*}
being $C>0$ sufficiently large. Let us consider the linear problem
\begin{equation}
\label{linear}
\left\{
\begin{alignedat}{2}
-\Delta v &= \hat{f}(x,u(x)) \quad &&\mbox{in} \;\; \Omega, \\
v &= 0 \quad &&\mbox{on} \;\; \partial \Omega.
\end{alignedat}
\right.
\end{equation}
Problem \eqref{linear} admits a unique solution $ v\in C^{1,\tau}_0(\overline{\Omega}) $, for some $ \tau \in (0,1] $, by virtue of Minty-Browder's theorem, Hardy's inequality, and \cite[Lemma 3.1]{H}. It turns out that the problem
\begin{equation}
\label{perturb}
\left\{
\begin{alignedat}{2}
-\Div (a(|\nabla w|)\nabla w-\nabla v(x)) &= 0 \quad &&\mbox{in} \;\; \Omega, \\
w &= 0 \quad &&\mbox{on} \;\; \partial \Omega,
\end{alignedat}
\right.
\end{equation}
admits a unique solution $w \in C^{1,\tau}_0(\overline{\Omega})$, by means of Minty-Browder's theorem and Lieberman's regularity theory, jointly with $\nabla v \in C^{0,\tau}(\overline{\Omega})$. Since $u$ is a solution to \eqref{perturb}, by uniqueness we get $u=w$, and thus $u \in C^{1,\tau}_0(\overline{\Omega})$.
\end{proof}

\section{Existence and multiplicity results}

In this last section we produce some results about \eqref{lambdaprob}. Lemma \ref{subsollemma} furnishes a sub-solution (depending on $\lambda$) to \eqref{lambdaprob}. Thus, taking into account Lemma \ref{equivalence}, the solutions to problem
\begin{equation}
\label{lambdaauxprob}
\tag{${\rm P}_{\lambda,\hat{f}}$}
\left\{
\begin{alignedat}{2}
-\Delta_\Phi u &= \lambda\hat{f}(x,u) \quad &&\mbox{in} \;\; \Omega, \\
u &> 0 \quad &&\mbox{in} \;\; \Omega, \\
u &= 0 \quad &&\mbox{on} \;\; \partial \Omega, 
\end{alignedat}
\right.
\end{equation}
are exactly the ones of \eqref{lambdaprob}. The energy functional associated with \eqref{lambdaauxprob} is
\begin{equation}
\label{Jlambda}
J_\lambda := H-\lambda K,
\end{equation}
being $H,K$ as in Lemmas \ref{princprop} and \ref{reactprop}, respectively. So the solutions to \eqref{lambdaauxprob} are the critical points of $J_\lambda$.

Existence of a solution is guaranteed by \cite[Theorem 2.1]{CGP}, that we report below. This result traces back to \cite{BC,Bo}.
\begin{thm}
\label{varprinc}
Let $X$ be a reflexive Banach space, $H:X \to \R$ and $K:X \to \R$ be two continuously G\^{a}teaux differentiable functionals such that $H$ is coercive and sequentially weakly lower semi-continuous, while $K$ is sequentially weakly upper semi-continuous with $\inf_X H = H(0) = K(0)$, and $r>0$. Then, for every
\begin{equation}
\label{ratio}
\lambda \in \left] 0,\frac{r}{\sup_{H^{-1}([0,r])} K} \right[,
\end{equation}
the functional $J_\lambda := H-\lambda K$ has a critical point $u_\lambda \in H^{-1}([0,r])$ satisfying $J_\lambda(u_\lambda) \leq J_\lambda(v)$ for all $v \in H^{-1}([0,r])$.
\end{thm}

A second solution is furnished by the Mountain Pass theorem (vide, e.g., \cite[Theorem 5.40]{MMP}): the applicability of this result relies, in our context, on the Palais-Smale condition.
\begin{defi}[${\rm (PS)}$]
\label{PS}
Let $ X $ be a Banach space and $ J \in C^1(X) $. We say that $ J $ satisfies the Palais-Smale condition if any sequence $ \{u_n\} \subseteq X $ such that $ \{J(u_n)\} $ is bounded and $ \|J'(u_n)\|_{X^*} \to 0 $ admits a convergent subsequence.
\end{defi}
\begin{thm}
\label{mountainpass}
Suppose $ X $ to be a Banach space, and $ J \in C^1(X) $ satisfying $ {\rm (PS)} $. Let $ u_0, u_1 \in X $, and $ \rho > 0 $ such that
\begin{equation}
\label{mpgeometry}
\max\{J(u_0),J(u_1)\} \leq \inf_{\partial B(u_0,\rho)} J =: \eta_\rho, \quad \|u_1-u_0\|_X > \rho.
\end{equation}
Set
\[
\Gamma := \left\{ \gamma \in C^0([0,1];X): \, \gamma(0) = u_0, \, \gamma(1) = u_1 \right\}, \quad c := \inf_{\gamma \in \Gamma} \sup_{t \in [0,1]} J(\gamma(t)).
\]
Then $ c \geq \eta_\rho $ and there exists $ u \in X $ such that $ J(u) = c $ and $ J'(u) = 0 $. Moreover, if $ c = \eta_\rho $, then $ u $ can be taken on $ \partial B(u_0,\rho) $.
\end{thm}

The next theorem concerns the existence of a solution to \eqref{lambdaprob}.

\begin{thm}
\label{firstsol}
Suppose ${\rm H(a)_1}$--${\rm H(a)_2}$ and ${\rm H(f)_1}$-${\rm H(f)_2}$. Then there exists $\lambda^* \in (0,+\infty]$ such that, for all $\lambda \in (0,\lambda^*)$, problem \eqref{lambdaprob} admits a solution $u_\lambda \in C^{1,\tau}_0(\overline{\Omega})$, being $\tau \in (0,1]$ opportune. Moreover, there exists $r^*_\lambda>0$ (depending on $\lambda \in (0,\lambda^*)$) such that $\int_\Omega \Phi(|\nabla u_\lambda|) \dx < r^*_\lambda$.
\end{thm}

\begin{proof}
We want to apply Theorem \ref{varprinc} to $J_\lambda$ (see \eqref{Jlambda}). As observed above, this theorem furnishes a solution $u_\lambda \in W^{1,\Phi}_0(\Omega)$ to \eqref{lambdaprob}. Then the regularity of $u_\lambda$ is a consequence of Theorem \ref{holder}.

In order to bound from above the ratio
\begin{equation*}
\frac{\sup_{H^{-1}([0,r])} K}{r},
\end{equation*}
appearing in \eqref{ratio}, we exploit \eqref{Fest} and estimate $K$ as follows:
\begin{equation}
\label{boundest}
|K(u)| \leq \int_\Omega |\hat{F}(x,u)| \dx \leq C_1|\Omega| + C_2\int_\Omega \Upsilon(|u|) \dx.
\end{equation}
Thus, we are led to study the function $\kappa:(0,+\infty) \to (0,+\infty)$ defined as
\begin{equation*}
\kappa(r) := \frac{C_1 |\Omega|}{r} + \frac{C_2}{r} \sup \left\{\int_\Omega \Upsilon(|u|) \dx: \, \int_\Omega \Phi(|\nabla u|) \dx \leq r \right\}.
\end{equation*}
Notice that $\kappa \to +\infty$ as $r \to 0^+$. Now we distinguish four cases, depending whether $\Upsilon \ll \Phi$, $\Upsilon < \Phi$, $\Upsilon > \Phi$, or $\Upsilon \gg \Phi $. Clearly, some cases overlap.

\noindent
\textit{First case}: $\Upsilon \ll \Phi$. \\
Fix an arbitrary $\epsilon \in (0,1]$. There exists $M_\epsilon>0$ such that
\begin{equation*}
\Upsilon(t) \leq \Phi(\epsilon t) \leq \epsilon \Phi(t) \quad \forall t > M_\epsilon.
\end{equation*}
Hence, by Poincaré's inequality \cite[p.8]{CGSS} and \eqref{factor}, we get
\begin{equation*}
\begin{split}
\int_\Omega \Upsilon(|u|) \dx &= \int_{\Omega \cap \{|u| \leq M_\epsilon\}} \Upsilon(|u|) \dx + \int_{\Omega \cap \{|u| > M_\epsilon\}} \Upsilon(|u|) \dx \\
&\leq \Upsilon(M_\epsilon)|\Omega| + \epsilon \int_\Omega \Phi(|u|) \dx \\
&\leq \Upsilon(M_\epsilon)|\Omega| + \epsilon \overline{\zeta}_\Phi(2d_\Omega) \int_\Omega \Phi(|\nabla u|) \dx \\
&\leq \Upsilon(M_\epsilon)|\Omega| + \epsilon \overline{\zeta}_\Phi(2d_\Omega) r.
\end{split}
\end{equation*}
Thus $\kappa$ can be estimated as
\begin{equation}
\label{kappafirst}
\kappa(r) \leq \frac{(C_1+C_2\Upsilon(M_\epsilon))|\Omega|}{r} + C_2 \overline{\zeta}_\Phi(2d_\Omega)\epsilon.
\end{equation}
Notice that the right-hand side of \eqref{kappafirst} is decreasing in $r$. Moreover, $\kappa(r) \to 0$ as $r \to +\infty$: indeed, letting $r \to +\infty$ in \eqref{kappafirst} reveals that
\begin{equation*}
\limsup_{r \to +\infty} \kappa(r) \leq C_2\overline{\zeta}_\Phi(2d_\Omega)\epsilon \quad \forall \epsilon \in (0,1],
\end{equation*}
since $\epsilon$ was arbitrary. We set $\lambda^*=+\infty$. Then, for any $\lambda>0$, we choose $\epsilon = \min\{1,(2C_2\overline{\zeta}_\Phi(2d_\Omega)\lambda)^{-1}\}$ and $r^*_\lambda>2\lambda(C_1+C_2\Upsilon(M_\epsilon))|\Omega|$. According to \eqref{kappafirst}, these choices guarantee $\kappa(r^*_\lambda) < \lambda^{-1}$, which allows to apply Theorem \ref{varprinc} with $r=r^*_\lambda$.

\noindent
\textit{Second case}: $\Upsilon < \Phi$. \\
There exist $ M,c>0$ such that
\begin{equation*}
\Upsilon(t) \leq \Phi(ct) \quad \forall t>M.
\end{equation*}
Reasoning as in the first case we have
\begin{equation*}
\begin{split}
\int_\Omega \Upsilon(|u|) \dx &\leq \Upsilon(M)|\Omega| + \int_\Omega \Phi(c|u|) \dx \leq \Upsilon(M)|\Omega| + \overline{\zeta}_\Phi(2cd_\Omega) \int_\Omega \Phi(|\nabla u|) \dx \\
&\leq \Upsilon(M)|\Omega| + \overline{\zeta}_\Phi(2cd_\Omega)r.
\end{split}
\end{equation*}
In this case $\kappa$ can be estimated as
\begin{equation}
\label{kappatwo}
\kappa(r) \leq \frac{(C_1+C_2\Upsilon(M))|\Omega|}{r} + C_2 \overline{\zeta}_\Phi(2cd_\Omega).
\end{equation}
We observe that the right-hand side of \eqref{kappatwo} is decreasing in $r$ and
\begin{equation*}
\limsup_{r \to +\infty} \kappa(r) \leq C_2 \overline{\zeta}_\Phi(2cd_\Omega).
\end{equation*}
We set $\lambda^* = (C_2 \overline{\zeta}_\Phi(2cd_\Omega))^{-1}$ and, for any $\lambda \in (0,\lambda^*)$, we take $r^*_\lambda > \frac{(C_1+C_2\Upsilon(M))|\Omega|}{\lambda^{-1}-C_2 \overline{\zeta}_\Phi(2cd_\Omega)}$. Then one applies Theorem \ref{varprinc}.

\noindent
\textit{Third case}: $\Upsilon > \Phi$. \\
Loosing information but not generality, we can reduce to the next case, namely, $\Upsilon \gg \Phi$. Indeed, in place of $\Upsilon$ in \eqref{Fest}, we can consider the intermediate function\footnote{Our definition should not be confused with the one in \cite[Definition 6.3.1]{RR}.} $\hat{\Upsilon} := \sqrt{\Upsilon \Phi_*}$.

Firstly, we notice that $\Upsilon \ll \Phi_*$ implies
\begin{equation*}
\Upsilon(t) \leq \Phi_*(t) \quad \forall t>M,
\end{equation*}
being $M>0$ opportune. So \eqref{Fest} can be re-written as
\begin{equation*}
|\hat{F}(x,s)| \leq C_1 + C_2 \Upsilon(M) + C_2 \hat{\Upsilon}(|s|) =: \hat{C}_1 + \hat{C}_2 \hat{\Upsilon}(|s|).
\end{equation*}

Secondly, it is readily seen that $\Upsilon \ll \hat{\Upsilon} \ll \Phi_*$ since, for any fixed $\eta>0$, by \eqref{factor} we have
\begin{equation*}
\frac{\hat{\Upsilon}(\eta t)}{\Upsilon(t)} = \sqrt{\frac{\Upsilon(\eta t)}{\Upsilon(t)}} \sqrt{\frac{\Phi_*(\eta t)}{\Upsilon(t)}} \geq \sqrt{\underline{\zeta}_\Upsilon(\eta)} \sqrt{\frac{\Phi_*(\eta t)}{\Upsilon(t)}} \xrightarrow{t \to +\infty} +\infty
\end{equation*}
and
\begin{equation*}
\frac{\Phi_*(\eta t)}{\hat{\Upsilon}(t)} = \sqrt{\frac{\Phi_*(\eta t)}{\Phi_*(t)}} \sqrt{\frac{\Phi_*(\eta t)}{\Upsilon(t)}} \geq \sqrt{\underline{\zeta}_{\Phi_*}(\eta)} \sqrt{\frac{\Phi_*(\eta t)}{\Upsilon(t)}} \xrightarrow{t \to +\infty} +\infty.
\end{equation*}
In particular we get $\Phi \ll \hat{\Upsilon}$.

Finally, we notice the $i_{\hat{\Upsilon}},s_{\hat{\Upsilon}}$ are well defined as in \eqref{indices}: indeed,
\begin{equation*}
\frac{t\hat{\Upsilon}'(t)}{\hat{\Upsilon}(t)} = \frac{1}{2} \frac{t\Upsilon'(t)}{\Upsilon(t)} + \frac{1}{2} \frac{t\Phi_*'(t)}{\Phi_*(t)} \quad \forall t>0.
\end{equation*}
We deduce $\frac{i_\Upsilon+i_{\Phi_*}}{2} \leq i_{\hat{\Upsilon}} \leq s_{\hat{\Upsilon}} \leq \frac{s_\Upsilon+s_{\Phi_*}}{2}$.

\noindent
\textit{Fourth case}: $\Upsilon \gg \Phi$. \\
Thanks to \eqref{comparison} and the embedding $W^{1,\Phi}_0(\Omega) \hookrightarrow L^\Upsilon(\Omega)$ we obtain
\begin{equation*}
\begin{split}
\int_\Omega \Upsilon(|u|) \dx &\leq \overline{\zeta}_\Upsilon(\|u\|_{L^\Upsilon(\Omega)}) \leq \overline{\zeta}_\Upsilon(k\|\nabla u\|_{L^\Phi(\Omega)}) \leq \overline{\zeta}_\Upsilon(k) \overline{\zeta}_\Upsilon(\|\nabla u\|_{L^\Phi(\Omega)}) \\
&= \overline{\zeta}_\Upsilon(k) \overline{\zeta}_\Upsilon(\underline{\zeta}_\Phi^{-1}(\underline{\zeta}_\Phi(\|\nabla u\|_{L^\Phi(\Omega)}))) \leq \overline{\zeta}_\Upsilon(k) \overline{\zeta}_\Upsilon \left(\underline{\zeta}_\Phi^{-1} \left(\int_\Omega \Phi(|\nabla u|) \dx \right) \right) \\
&\leq \overline{\zeta}_\Upsilon(k) \overline{\zeta}_\Upsilon(\underline{\zeta}_\Phi^{-1}(r)) \leq \overline{\zeta}_\Upsilon(k) (1 +  r^{\frac{s_\Upsilon}{i_\Phi}}),
\end{split}
\end{equation*}
where $k>0$ is the best constant of the embedding mentioned above. So $\kappa$ is estimated as
\begin{equation}
\label{kappafour}
\kappa(r) \leq \frac{C_1 |\Omega| + C_2 \overline{\zeta}_\Upsilon(k)}{r} + C_2 \overline{\zeta}_\Upsilon(k) r^{\frac{s_\Upsilon}{i_\Phi}-1}.
\end{equation}
We observe that $\Upsilon \gg \Phi$ implies $s_\Upsilon > i_\Phi$; otherwise we have
\begin{equation*}
\frac{\Upsilon'(s)}{\Upsilon(s)} \leq \frac{\Phi'(s)}{\Phi(s)} \quad \forall s \in (0,+\infty)
\end{equation*}
whence, integrating in $[1,t]$, $t>1$, and passing to the exponential,
\begin{equation*}
\Upsilon(t) \leq \frac{\Upsilon(1)}{\Phi(1)} \Phi(t) \quad \forall t \in (1,+\infty),
\end{equation*}
in contrast with $\Upsilon \gg \Phi$. Hence the right-hand side of \eqref{kappafour}, which can be re-written as
\begin{equation*}
\hat{k}(r) := \frac{A}{r}+Br^{\theta}, \quad \mbox{with} \quad A := C_1 |\Omega| + C_2\overline{\zeta}_\Upsilon(k), \;\; B:= C_2 \overline{\zeta}_\Upsilon(k), \;\; \theta := \frac{s_\Upsilon}{i_\Phi}-1 > 0,
\end{equation*}
diverges when $r \to +\infty$. Computing the unique critical point of $\hat{k}$ reveals that
\begin{equation*}
\min_{r>0} \hat{k}(r) = \hat{k} \left( \left( \frac{A}{\theta B} \right)^{\frac{1}{\theta+1}} \right) = \left[A^\theta B (\theta+\theta^{-\theta})\right]^{\frac{1}{\theta+1}}.
\end{equation*}
In this case we set $\lambda^* := \left[A^\theta B (\theta+\theta^{-\theta})\right]^{-\frac{1}{\theta+1}}$, $r^*_\lambda := \left( \frac{A}{\theta B} \right)^{\frac{1}{\theta+1}}$ and apply Theorem \ref{varprinc}.
\end{proof}

\begin{rmk}
\label{locmin}
According to \eqref{comparison} and Theorem \ref{firstsol}, we infer
\begin{equation*}
\underline{\zeta}_\Phi(\|u_\lambda\|_{W^{1,\Phi}_0(\Omega)}) \leq \int_\Omega \Phi(|\nabla u_\lambda|) \dx < r^*_\lambda.
\end{equation*}
We define the ball
\begin{equation*}
B_\lambda := \{u \in W^{1,\Phi}_0(\Omega): \, \|u\|_{W^{1,\Phi}_0(\Omega)} < \underline{\zeta}_\Phi^{-1}(r^*_\lambda)\}.
\end{equation*}
Taking into account Theorem \ref{firstsol} again, we deduce that $u_\lambda$ is a minimizer for the restriction of $J_\lambda$ to $\overline{B}_\lambda$; in particular, $u_\lambda$ is a local minimizer for $J_\lambda$. Incidentally, we stress the fact that this local minimizer has been provided without using any $W^{1,\Phi}$ versus $C^1$ local minimizer argument.
\end{rmk}

\begin{lemma}
\label{palaissmale}
Under ${\rm H(a)_1}$--${\rm H(a)_2}$ and ${\rm H(f)_1}$--${\rm H(f)_3}$, the functional $J_\lambda$ in \eqref{Jlambda} satisfies the Palais-Smale condition and is unbounded from below.
\end{lemma}
\begin{proof}
Let $\{u_n\} \subseteq W^{1,\Phi}_0(\Omega)$ be such that $\{J_\lambda(u_n)\}$ is bounded and $\|J'_\lambda(u_n)\|_{W^{-1,\overline{\Phi}}(\Omega)} \to 0$ as $n\to\infty$. Hence, for a suitable $c>0$, up to subsequences we have
\begin{equation}
\label{PS1}
\int_\Omega \Phi(|\nabla u_n|) \dx -\lambda \int_\Omega \hat{F}(x,u_n) \dx \leq c
\end{equation}
and
\begin{equation}
\label{PS2}
\left| \int_\Omega a(|\nabla u_n|)\nabla u_n \cdot \nabla v \dx -\lambda \int_\Omega \hat{f}(x,u_n)v \dx \right| \leq \|\nabla v\|_{L^\Phi(\Omega)}
\end{equation}
for all $n \in \N$ and $v \in W^{1,\Phi}_0(\Omega)$. We prove that $\{u_n\}$ is bounded in $W^{1,\Phi}_0(\Omega)$ by showing the boundedness of $\{u_n^-\}$ and $\{u_n^+\}$.

Choosing $v=-u_n^-$ in \eqref{PS2} and using \eqref{coerc} yield
\begin{equation*}
\begin{split}
&i_\Phi \underline{\zeta}_\Phi(\|\nabla u_n^-\|_{L^\Phi(\Omega)}) \leq \int_\Omega a(|\nabla u_n^-|)|\nabla u_n^-|^2 \dx \\
&\leq \int_\Omega a(|\nabla u_n^-|)|\nabla u_n^-|^2 \dx + \lambda \int_\Omega \hat{f}(x,u_n) u_n^- \dx \leq \|\nabla u_n^-\|_{L^\Phi(\Omega)},
\end{split}
\end{equation*}
whence $\{u_n^-\}$ is bounded in $W^{1,\Phi}_0(\Omega)$.

Exploiting \eqref{compact} and ${\rm H(f)_3}$, besides Remark \ref{deltaR}, we have
\begin{equation}
\label{FestAR}
\begin{split}
\int_\Omega \hat{F}(x,u_n^+) \dx &= \int_{\Omega\cap\{u_n^+\leq R\}} \hat{F}(x,u_n^+) \dx + \int_{\Omega\cap\{u_n^+>R\}} \left( \hat{F}(x,R) + F(x,u_n^+) \right) \dx \\
&\leq 2\Pi(R) + \int_{\Omega\cap\{u_n^+>R\}} F(x,u_n^+) \dx \\
&\leq 2\Pi(R) + \frac{1}{\mu} \int_{\Omega\cap\{u_n^+>R\}} f(x,u_n^+) u_n^+ \dx \\
&\leq 2\Pi(R) + \frac{1}{\mu} \int_\Omega \hat{f}(x,u_n^+) u_n^+ \dx.
\end{split}
\end{equation}
By \eqref{PS1} and \eqref{FestAR} we deduce
\begin{equation}
\label{PS1rev}
\begin{split}
\int_\Omega \Phi(|\nabla u_n^+|) \dx &\leq \int_\Omega \Phi(|\nabla u_n|) \dx \leq c + \lambda \int_\Omega \hat{F}(x,u_n) \dx \leq c + \lambda \int_\Omega \hat{F}(x,u_n^+) \dx \\
&\leq c + 2\lambda\Pi(R) + \frac{\lambda}{\mu} \int_\Omega \hat{f}(x,u_n^+) u_n^+ \dx.
\end{split}
\end{equation}
On the other hand, choosing $v=u_n^+$ in \eqref{PS2} produces
\begin{equation}
\label{PS2rev}
\begin{split}
\lambda \int_\Omega \hat{f}(x,u_n^+) u_n^+ \dx &\leq \|\nabla u_n^+\|_{L^\Phi(\Omega)} + \int_\Omega \phi(|\nabla u_n^+|) |\nabla u_n^+| \dx \\
&\leq \|\nabla u_n^+\|_{L^\Phi(\Omega)} + s_\Phi \int_\Omega \Phi(|\nabla u_n^+|) \dx.
\end{split}
\end{equation}
Combining \eqref{PS1rev}--\eqref{PS2rev} and rearranging the terms we obtain
\begin{equation*}
\left(1-\frac{s_\Phi}{\mu}\right) \int_\Omega \Phi(|\nabla u_n^+|) \dx \leq c + 2\lambda\Pi(R) + \frac{1}{\mu} \|\nabla u_n^+\|_{L^\Phi(\Omega)}.
\end{equation*}
According to ${\rm H(f)_3}$ and \eqref{comparison}, it turns out that $\{u_n^+\}$ is bounded in $W^{1,\Phi}_0(\Omega)$. The ${\rm (S_+)}$ property of $H'$ (see Lemma \ref{princprop}) and the compactness of $K'$ (see Lemma \ref{reactprop}) ensure the Palais-Smale condition for $J_\lambda$; see \cite[Lemma 3.1]{CGP} for details.

Now we prove that $J_\lambda$ is unbounded from below. Firstly, fix any $\overline{R}>R$. Integrating \eqref{AR+} in $(\overline{R},t)$, $t>\overline{R}$, and passing to the exponential yield
\begin{equation}
\label{FgrowthAR}
F(x,t) \geq \frac{F(x,\overline{R})}{\overline{R}^\mu} t^\mu =: c_{\overline{R}} t^\mu \quad \forall (x,t) \in \Omega \times [\overline{R},+\infty).
\end{equation}
Take any test function $u_0 \in C^\infty_c(\Omega)$ such that $u_0 \geq 0$ in $\Omega$ and $u_0 \not\equiv 0$. Then there exists a compact $K \subseteq \Omega$ such that
\begin{equation}
\label{u0props}
\min_K u_0 > 0 \quad \mbox{and} \quad |K|>0.
\end{equation}
For any $M>0$, set $K_M := \{x \in \Omega: \, Mu_0 > \overline{R}\}$. Observe that $\{K_M\}_{M>0}$ is increasing and $K \subseteq K_M$ for large values of $M$. Using \eqref{FgrowthAR} and \eqref{factor}, besides recalling Remark \ref{deltaR}, for $M$ large we get
\begin{equation}
\label{unbounded}
\begin{split}
J_\lambda(Mu_0) &\leq \int_\Omega \Phi(M|\nabla u_0|) \dx - \lambda \int_{K_M} F(x,Mu_0) \dx \\
&\leq \int_\Omega \Phi(M|\nabla u_0|) \dx - \lambda c_{\overline{R}} M^\mu \int_{K_M} u_0^\mu \dx \\
&\leq M^{s_\Phi} \int_\Omega \Phi(|\nabla u_0|) \dx - \lambda c_{\overline{R}} M^\mu \int_K u_0^\mu \dx.
\end{split}
\end{equation}
By ${\rm H(f)_3}$ we have $\mu > s_\Phi$, while \eqref{u0props} ensures that $\int_K u_0^\mu \dx>0$. Hence $J_\lambda(Mu_0) \to -\infty$ when $M \to +\infty$, as desired.
\end{proof}

\begin{rmk}
Incidentally, we notice that \eqref{FgrowthAR} implies that $J_\lambda$ is $\Phi$-super-linear, since $\Phi<t^{\mu}$ in the sense of \eqref{Youngcomp}.
\end{rmk}

\begin{thm}
\label{secondsol}
Suppose ${\rm H(a)_1}$--${\rm H(a)_2}$ and ${\rm H(f)_1}$--${\rm H(f)_3}$. Then problem \eqref{lambdaprob} admits two distinct solutions in $C^{1,\tau}_0(\overline{\Omega})$.
\end{thm}

\begin{proof}
Let $\lambda^*,r^*$ be given by Theorem \ref{firstsol}. Fix any $\lambda \in (0,\lambda^*)$. Existence of a solution $u_\lambda \in C^{1,\tau}_0(\Omega)$ to \eqref{lambdaprob} is guaranteed by Theorem \ref{firstsol}. We want to get a second solution $v_\lambda \in C^{1,\tau}_0(\Omega)$ by applying Theorem \ref{mountainpass} to the functional $J_\lambda$ defined in \eqref{Jlambda}. As in the proof of Theorem \ref{firstsol}, regularity of $v_\lambda$ is a consequence of Theorem \ref{holder}.

First we notice that $J_\lambda$ is bounded on bounded sets: indeed, by \eqref{comparison}, \eqref{boundest}, and the embedding inequality for $W^{1,\Phi}_0(\Omega) \hookrightarrow L^{\Upsilon}(\Omega)$ we have, for an opportune $C_3>0$ independent of $u$,
\begin{equation*}
\begin{split}
|J_\lambda(u)| &\leq \overline{\zeta}_\Phi(\|u\|_{W^{1,\Phi}_0(\Omega)}) + C_1|\Omega| + C_2 \overline{\zeta}_\Upsilon(\|u\|_{L^\Upsilon(\Omega)}) \\
&\leq \overline{\zeta}_\Phi(\|u\|_{W^{1,\Phi}_0(\Omega)}) + C_1|\Omega| + C_3 \overline{\zeta}_\Upsilon(\|u\|_{W^{1,\Phi}_0(\Omega)}).
\end{split}
\end{equation*}
Taking into account Remark \ref{locmin}, we have that $u_\lambda$ is a local minimizer for $J_\lambda$. Since Theorem \ref{palaissmale} ensures that $J_\lambda$ is unbounded from below, then $u_\lambda$ is not a global minimizer. Reasoning as in the first part of the proof of \cite[Theorem 2.1]{Bo2} guarantees \eqref{mpgeometry} with $u_0:=u_\lambda$. Hence Theorem \ref{mountainpass} furnishes $v_\lambda \in W^{1,\Phi}_0(\Omega)$ critical point to $J_\lambda$, and thus solution to both \eqref{lambdaauxprob} and \eqref{lambdaprob}. Moreover $v_\lambda$ fulfills $J_\lambda(v_\lambda) \geq J_\lambda(u_\lambda)$. If $J_\lambda(v_\lambda) > J_\lambda(u_\lambda)$, then $v_\lambda \neq u_\lambda$; else, Theorem \ref{mountainpass} ensures that $v_\lambda$ can be taken on $\partial B_\lambda$. In any case we have $v_\lambda \neq u_\lambda$.
\end{proof}

\appendix

\section{Examples}

In this appendix we want to show the importance of working in Sobolev-Orlicz spaces instead of classical Sobolev ones. In this sight, we furnish a class of Young functions $\Phi$ and reaction terms $f$ whose corresponding problem \eqref{lambdaprob} cannot be set in a Sobolev framework, but it fulfills ${\rm H(a)_1}$--${\rm H(a)_2}$ and ${\rm H(f)_1}$--${\rm H(f)_3}$; see Example \ref{ex1} below. Inspiring examples can be found, e.g., in \cite{CGHMS}, where existence of at least one positive solution is obtained by the Mountain Pass theorem provided $\lambda=1$ and the reaction term is not affected by singular terms. Following \cite[Example 1]{CGHMS}, at the end of this appendix we will furnish a more concrete example (vide Example \ref{ex2} below) satisfying our hypotheses.

First of all, we construct a class of `pathological' Young functions $\Psi$ (with $1<i_\Psi<s_\Psi<+\infty$), possessing distinct indices at infinity, that is,
\begin{equation*}
\liminf_{t\to+\infty} \frac{t\Psi'(t)}{\Psi(t)} \neq \limsup_{t\to+\infty} \frac{t\Psi'(t)}{\Psi(t)}, \quad \liminf_{t\to+\infty} \frac{t\Psi''(t)}{\Psi'(t)} \neq \limsup_{t\to+\infty} \frac{t\Psi''(t)}{\Psi'(t)}.
\end{equation*}
This could be hopefully useful also in other contexts to construct counterexamples in Orlicz spaces.

\begin{lemma}
\label{example}
Let $1<q<p<+\infty$. Set $\alpha := \frac{p+q}{2}$ and $\beta := \frac{p-q}{2}$. Then, for any $\epsilon<\min\{4,\frac{q-1}{\beta}\}$, the function
\begin{equation}
\label{young}
\Psi(t) := t^\alpha e^{\eta(t)} \quad \forall t \geq 0,
\end{equation}
with
\begin{equation}
\label{young2}
\eta(t) = \left\{
\begin{alignedat}{2}
&\frac{\beta \epsilon}{2e^2} (e-t)^2 - \frac{\beta\epsilon}{1+\epsilon^2} \quad &&\mbox{for} \;\; 0 \leq t \leq e, \\
&\beta \frac{\log t}{1+\epsilon^2} [\sin(\epsilon \log (\log t))-\epsilon \cos(\epsilon \log (\log t))] \quad &&\mbox{for} \;\; t \geq e,
\end{alignedat}
\right. 
\end{equation}
is a Young function satisfying the following properties:
\begin{equation}
\label{index1}
\inf_{t>0} \frac{t\Psi'(t)}{\Psi(t)} = \liminf_{t\to+\infty} \frac{t\Psi'(t)}{\Psi(t)} = q < p = \limsup_{t\to+\infty} \frac{t\Psi'(t)}{\Psi(t)} = \sup_{t>0} \frac{t\Psi'(t)}{\Psi(t)},
\end{equation}
\begin{equation}
\label{index2}
\begin{split}
&q-1-\beta\epsilon < \inf_{t>0} \frac{t\Psi''(t)}{\Psi'(t)} \leq \liminf_{t\to+\infty} \frac{t\Psi''(t)}{\Psi'(t)} = q-1 \\
&< p-1 = \limsup_{t\to+\infty} \frac{t\Psi''(t)}{\Psi'(t)} \leq \sup_{t>0} \frac{t\Psi''(t)}{\Psi'(t)} < p-1+\beta\epsilon,
\end{split}
\end{equation}
\begin{equation}
\label{notpower}
\liminf_{t \to +\infty} \frac{\Psi(t)}{t^r} = 0 \quad \mbox{or} \quad \limsup_{t \to +\infty} \frac{\Psi(t)}{t^r} =+\infty \quad \forall r>1.
\end{equation}
If $p<N$ then $\Psi$ satisfies \eqref{plessN} with $\Psi$ in place of $\Phi$. If, in addition, $p<q^*$, then $s_\Psi<i_{\Psi_*}$.
\end{lemma}

\begin{proof}
Starting from \eqref{young}, let us compute $\Psi'$, $\Psi''$ in terms of the lower order derivatives:
\begin{equation}
\label{der1}
\Psi'(t) = \Psi(t) \left( \frac{\alpha}{t} + \eta'(t) \right) = \frac{\Psi(t)}{t}(\alpha+t\eta'(t)), 
\end{equation}
\begin{equation}
\label{der2}
\begin{split}
\Psi''(t) &= \Psi'(t) \left( \frac{\Psi'(t)}{\Psi(t)} + \frac{\eta''(t)-\frac{\alpha}{t^2}}{\frac{\alpha}{t} + \eta'(t)} \right) = \frac{\Psi'(t)}{t} \left( \alpha+t\eta'(t) + \frac{t^2\eta''(t)-\alpha}{\alpha+t\eta'(t)} \right) \\
&= \frac{\Psi'(t)}{t} \left( \alpha-1+t\eta'(t) + \frac{t^2\eta''(t)+t\eta'(t)}{\alpha+t\eta'(t)} \right).
\end{split}
\end{equation}
Firstly we study $\Psi$ in the interval $[0,e]$. We have
\begin{equation}
\label{etaprime1}
\eta'(t) = \frac{\beta\epsilon}{e^2}(t-e) \quad \mbox{and} \quad \eta''(t) = \frac{\beta\epsilon}{e^2} \quad \mbox{for all} \;\; t \in (0,e].
\end{equation}
We observe that \eqref{der1}, \eqref{young2}, and $\epsilon < 4$ entail
\begin{equation}
\label{0e1}
q < \alpha - \frac{\beta\epsilon}{4} = \alpha+\min_{s\in(0,e]} s\eta'(s) \leq \frac{t\Psi'(t)}{\Psi(t)} \leq \alpha+\max_{s \in (0,e]} s\eta'(s) = \alpha < p
\end{equation}
for all $t\in(0,e]$. Exploiting \eqref{der2} and \eqref{0e1}, $\epsilon<4$, the monotonicity of $r \mapsto r+\frac{r}{\alpha+r}$, and $\eta'<0<\eta''$ in $(0,e]$, we obtain, for all $t \in (0,e]$,
\begin{equation}
\label{0e2}
\begin{split}
\alpha-1-\beta\epsilon &< \alpha-1-\frac{\beta\epsilon}{4}\left( 1+\frac{1}{\alpha-\frac{\beta\epsilon}{4}} \right) \\
&= \alpha-1+\min_{s\in(0,e]} \left( s\eta'(s)+\frac{s\eta'(s)}{\alpha+s\eta'(s)} \right) \\
&\leq \frac{t\Psi''(t)}{\Psi'(t)} \leq \alpha-1+\frac{t^2\eta''(t)}{\alpha+t\eta'(t)} \\
&\leq \alpha-1+\frac{\beta\epsilon}{\alpha-\frac{\beta\epsilon}{4}} < \alpha-1+\beta\epsilon.
\end{split}
\end{equation}

Now we analyze $\Psi$ in $[e,+\infty)$. We posit $\zeta(t) := \epsilon(\log(\log t))$ for all $t \geq e$. Integrating by parts twice reveals that
\begin{equation*}
\begin{split}
\int \sin(\epsilon \log s) \ds &= s\sin(\epsilon \log s) - \epsilon \int \cos(\epsilon \log s) \ds \\
&= s[\sin(\epsilon \log s)-\epsilon\cos(\epsilon \log s)] - \epsilon^2 \int \sin(\epsilon \log s) \ds,
\end{split}
\end{equation*}
whence, performing the change of variable $s=\log t$ and recalling \eqref{young2},
\begin{equation}
\label{eta}
\begin{split}
\beta \int \frac{\sin(\zeta(t))}{t} \dt &= \beta\int \sin(\epsilon \log s) \ds = \frac{\beta s}{1+\epsilon^2}[\sin(\epsilon \log s)-\epsilon\cos(\epsilon \log s)] \\
&= \beta \frac{\log t}{1+\epsilon^2}[\sin(\epsilon \log(\log t))-\epsilon\cos(\epsilon \log(\log t))] = \eta(t)
\end{split}
\end{equation}
for all $t \in[e,+\infty)$. Accordingly, we have
\begin{equation}
\label{etaprime2}
\eta'(t) = \beta \frac{\sin(\zeta(t))}{t} \quad \mbox{and} \quad \eta''(t) = \frac{\beta}{t^2} [t\zeta'(t)\cos(\zeta(t))-\sin(\zeta(t))] \quad \mbox{for all} \;\; t \geq e.
\end{equation}
Hence we rewrite \eqref{der1}--\eqref{der2} as
\begin{equation}
\label{der1bis}
\Psi'(t) = \frac{\Psi(t)}{t} (\alpha+\beta\sin(\zeta(t))),
\end{equation}
\begin{equation}
\label{der2bis}
\Psi''(t) = \frac{\Psi'(t)}{t} \left( \alpha-1+\beta\sin(\zeta(t)) + \frac{\beta t\zeta'(t)\cos(\zeta(t))}{\alpha+\beta\sin(\zeta(t))} \right),
\end{equation}
valid for all $t \in [e,+\infty)$.

We observe that $\zeta(t) \to +\infty$ as $t \to +\infty$, so \eqref{0e1} and \eqref{der1bis} guarantee \eqref{index1}. Moreover, $t\zeta'(t) \to 0$ as $t \to +\infty$ and
\begin{equation}
\label{bounded}
0 \leq \frac{\beta |\cos(\zeta(t))|}{\alpha+\beta\sin(\zeta(t))} \leq \frac{\beta}{\alpha-\beta} \quad \forall t \in[e,+\infty).
\end{equation}
Thus, \eqref{der2bis} provides the equalities in \eqref{index2}. More precisely, we notice that
\begin{equation}
\label{bounded2}
0 < t\zeta'(t) = \frac{\epsilon}{\log t} \leq \epsilon \quad \forall t \in[e,+\infty).
\end{equation}
Exploiting \eqref{der2bis}--\eqref{bounded2} we get, for all $t \geq e$,
\begin{equation}
\label{convex}
q-1-\beta\epsilon < \alpha-1-\beta-\frac{\beta\epsilon}{\alpha-\beta} \leq \frac{t\Psi''(t)}{\Psi'(t)} \leq \alpha-1+\beta+\frac{\beta\epsilon}{\alpha-\beta} < p-1+\beta\epsilon.
\end{equation}
Because of \eqref{0e2} and \eqref{convex}, the inequalities in \eqref{index2} hold true.

A direct computation, based on \eqref{etaprime1} and \eqref{etaprime2}, shows that $\eta \in C^2(0,+\infty)$; thus, $\Psi$ enjoys the same property. Moreover, \eqref{index1}--\eqref{index2} and $\beta\epsilon<q-1$ yield $\Psi'(t),\Psi''(t)>0$ for all $t>0$. Thus $\Psi$ is strictly increasing and convex. Using again \eqref{index1}, together with \eqref{factor}, we deduce
\begin{equation*}
\Psi(1) \min\{t^p,t^q\} = \Psi(1) \underline{\zeta}_\Psi(t) \leq \Psi(t) \leq \Psi(1) \overline{\zeta}_\Psi(t) = \Psi(1) \max\{t^p,t^q\} \quad \forall t>0,
\end{equation*}
which entails \eqref{Nfunct}. Hence $\Psi$ is a Young function.

In order to prove \eqref{notpower}, let us consider two sequences $h_n,k_n \to +\infty$ such that
\begin{equation*}
\sin(\zeta(h_n)) = 1 \quad \mbox{and} \cos(\zeta(k_n)) = 1 \quad \mbox{for all} \;\; n \in \N.
\end{equation*}
Then \eqref{young}--\eqref{young2} give, for any $n$ large enough,
\begin{equation*}
\Psi(h_n) = h_n^{\alpha+\frac{\beta}{1+\epsilon^2}} \quad \mbox{and} \quad \Psi(k_n) = k_n^{\alpha-\frac{\beta\epsilon}{1+\epsilon^2}},
\end{equation*}
ensuring \eqref{notpower}.

Now suppose that $p<N$. Then, setting $\Lambda := t^p$, by \eqref{factor} we infer $\Psi<\Lambda$ in the sense of \eqref{Youngcomp}. In particular, $\Psi^{-1}(t) \geq c\Lambda^{-1}(t)$ for all $t>1$, being $c>0$ small enough. Thus we get
\begin{equation*}
\int_1^{+\infty} \Theta_{\Psi}(t) \dt \geq c\int_1^{+\infty} \Theta_\Lambda(t) \dt = c \int_1^{+\infty} t^{\frac{1}{p^*}-1} \dt = +\infty.
\end{equation*}
The last statement is a direct consequence of \eqref{index1} and \eqref{sobind}.
\end{proof}

\begin{rmk}
\label{settings}
Two motivations suggest to work in Sobolev-Orlicz spaces instead of in the classical Sobolev framework.

The first motivation is structural: if we set the problem in a reflexive Sobolev-Orlicz space $W^{1,\Psi}_0(\Omega)$ (which may be a Sobolev space), the weak formulation of problem \eqref{lambdaprob} requires
\begin{equation*}
\int_\Omega a(|\nabla u|) \nabla u \nabla v \dx < +\infty \quad \forall u,v \in W^{1,\Psi}_0(\Omega).
\end{equation*}
This is a \textit{duality} property, which fails whenever $W^{1,\Psi}_0(\Omega) \setminus W^{1,\Phi}_0(\Omega) \neq \emptyset$: indeed, taking $u \in W^{1,\Psi}_0(\Omega) \setminus W^{1,\Phi}_0(\Omega)$ and $v=u$, by \eqref{indices} we get
\begin{equation*}
\int_\Omega a(|\nabla u|)\nabla u \nabla v \dx = \int_\Omega \phi(|\nabla u|)|\nabla u| \dx \geq i_\Phi \int_\Omega \Phi(|\nabla u|) \dx = +\infty.
\end{equation*}
Hence, in order to properly define the concept of `weak solution', we have to require $W^{1,\Psi}_0(\Omega) \subseteq W^{1,\Phi}_0(\Omega)$, which means $\Phi<\Psi$ (in the sense of \eqref{Youngcomp}).

Here comes the second motivation, which is technical: if we suppose $W^{1,\Psi}_0(\Omega) \subsetneq W^{1,\Phi}_0(\Omega)$, then we loose the \textit{coercivity} of $-\Delta_\Phi$. To show this, we pick $u \in W^{1,\Phi}_0(\Omega) \setminus W^{1,\Psi}_0(\Omega)$ and a sequence $\{u_n\} \subseteq W^{1,\Psi}_0(\Omega)$ such that $u_n \to u$ in $W^{1,\Phi}_0(\Omega)$. It turns out that $\|u_n\|_{W^{1,\Psi}_0(\Omega)} \to +\infty$; otherwise, by reflexivity of $W^{1,\Psi}_0(\Omega)$ and up to subsequences, we would have $u_n \rightharpoonup u^*$ in $W^{1,\Psi}_0(\Omega)$ for some $u^* \in W^{1,\Psi}_0(\Omega)$ and, by uniqueness of the weak limit, we would conclude $u=u^* \in W^{1,\Psi}_0(\Omega)$, in contrast with the choice of $u$. On the other hand, by \eqref{indices},
\begin{equation*}
\begin{split}
&\sup_{n \in \N} \int_\Omega a(|\nabla u_n|) |\nabla u_n|^2 \dx = \sup_{n \in \N} \int_\Omega \phi(|\nabla u_n|)|\nabla u_n| \dx \\
&\leq s_\Phi \sup_{n \in \N} \int_\Omega \Phi(|\nabla u_n|) \dx \leq s_\Phi \sup_{n \in \N} \overline{\zeta}_\Phi(\|u_n\|_{W^{1,\Phi}_0(\Omega)}) <+\infty,
\end{split}
\end{equation*}
which proves that $-\Delta_\Phi$ is not coercive on $W^{1,\Psi}_0(\Omega)$. Since coercivity of the principal part is an essential ingredient for existence results and, in particular, for our approach, which relies on Theorem \ref{varprinc}, we adopted the framework $W^{1,\Psi}_0(\Omega) = W^{1,\Phi}_0(\Omega)$.

It remains to prove that $W^{1,\Phi}_0(\Omega)$ is not a Sobolev space in general. To this end, we observe that any Young function given by Lemma \ref{example}, say $\Phi$, furnishes a counterexample. Indeed, suppose by contradiction that $W^{1,\Phi}_0(\Omega) = W^{1,r}_0(\Omega)$ for some $r>1$. Then we have $\Phi<t^r$ and $t^r<\Phi$ (in the sense of \eqref{Youngcomp}), whence
\begin{equation}
\label{notpowerabs}
\begin{split}
\Phi<t^r &\Rightarrow \limsup_{t \to +\infty} \frac{\Phi(t)}{t^r} \leq c_1^r <+\infty,\\
t^r<\Phi &\Rightarrow \liminf_{t \to +\infty} \frac{\Phi(t)}{t^r} \geq c_2^{-r} >0,
\end{split}
\end{equation}
for a suitable $c_1,c_2>0$ given by \eqref{Youngcomp}. Since \eqref{notpowerabs} contradicts \eqref{notpower}, we deduce that $W^{1,\Phi}_0(\Omega)$ is not a Sobolev space.
\end{rmk}

\begin{rmk}
\label{reactrmk}
Another important aspect related to the choice of the Sobolev-Orlicz framework is represented by the reaction term: we address the reader to \cite[Section 6]{CGHMS} for a discussion about this setting and the Ambrosetti-Rabinowitz condition. Here we limit ourselves to provide an example of nonlinearity $f=f(u)$ fulfilling ${\rm H(f)_1}$--${\rm H(f)_3}$.

Suppose $s_\Phi<i_{\Phi_*}$. By virtue of Lemma \ref{example}, we can construct a Young function $\Upsilon$ satisfying $s_\Phi<i_\Upsilon<s_\Upsilon<i_{\Phi_*}$. Then, fixed $\gamma \in (0,1)$, we consider
\begin{equation}
\label{reaction}
f(t) = \frac{\Upsilon(t)}{t} + t^{-\gamma}.
\end{equation}
Obviously, $f$ fulfills ${\rm H(f)_1}$. Observe that \eqref{equivalent} implies
\begin{equation*}
\frac{\Upsilon(t)}{t} \leq \overline{\Upsilon}^{-1}(\Upsilon(t)) \quad \forall t>0,
\end{equation*}
so that ${\rm H(f)_2}$ is satisfied with $c_1=c_2=1$. To prove ${\rm H(f)_3}$, choose any $\mu \in (s_\Phi,i_\Upsilon)$. For all $t>0$ we have
\begin{equation}
\label{RHS}
tf(t) = \Upsilon(t)+t^{1-\gamma}
\end{equation}
and, given any $R>0$,
\begin{equation}
\label{LHS}
F(t) = \int_R^t \left(\frac{\Upsilon(s)}{s}+s^{-\gamma}\right) \ds \leq  \int_R^t (i_\Upsilon^{-1}\Upsilon'(s) + s^{-\gamma}) \ds \leq \frac{1}{i_\Upsilon} \Upsilon(t) + \frac{t^{1-\gamma}}{1-\gamma}.
\end{equation}
Convexity of $\Upsilon$ and $\mu<i_\Upsilon$ guarantee that there exists $R>0$ such that
\begin{equation}
\label{convexity}
\frac{\mu}{1-\gamma} t^{1-\gamma} \leq \left(1-\frac{\mu}{i_\Upsilon}\right) \Upsilon(t) \quad \forall t \geq R.
\end{equation}
From \eqref{RHS}--\eqref{convexity} we get
\begin{equation*}
\mu F(t) \leq \frac{\mu}{i_\Upsilon} \Upsilon(t) + \frac{\mu}{1-\gamma} t^{1-\gamma} \leq \Upsilon(t) \leq tf(t) \quad \forall t \geq R,
\end{equation*}
which entails ${\rm H(f)_3}$.
\end{rmk}

As announced, we conclude with two examples of problems fulfilling the hypotheses of Theorem \ref{secondsol}; according to Remark \ref{settings}, we stress that it is necessary to set them in the appropriate Sobolev-Orlicz setting. Existence of two solutions for these problems is a consequence of Theorem \ref{secondsol}.
\begin{ex}
\label{ex1}
Take any $r>s>p>q>1$ such that $p<N$ and $r<q^*$. Let $\Phi$ and $\Upsilon$ be given by Lemma \ref{example} (applied with any sufficiently small $\epsilon>0$), such that $i_\Phi=q$, $s_\Phi=p$, $i_\Upsilon=s$, $s_\Upsilon=r$. Let $f$ be defined as in \eqref{reaction}. Then problem \eqref{lambdaprob} admits at least two distinct weak solutions $u,v \in C^{1,\tau}_0(\overline{\Omega})$ for all $\lambda \in (0,\lambda^*)$. Here $\tau \in (0,1]$ and $\lambda^*>0$ are given by Theorems \ref{holder} and \ref{firstsol}, respectively. \\
The hypotheses of Theorem \ref{secondsol} are fulfilled: \eqref{index2} implies ${\rm H(a)_1}$ and the final part of Lemma \ref{example} gives ${\rm H(a)_2}$, while Remark \ref{reactrmk} ensures ${\rm H(f)_1}$--${\rm H(f)_3}$.
\end{ex}
\begin{ex}
\label{ex2}
The same result stated in Example \ref{ex1} holds true for the problem
\begin{equation}
\label{concrete}
\left\{
\begin{alignedat}{2}
-\Div(\log(1+|\nabla u|)|\nabla u|^{p-2}\nabla u) &= \lambda (u^r+u^{-\gamma}) \quad &&\mbox{in} \;\; \Omega, \\
u &> 0 \quad &&\mbox{in} \;\; \Omega, \\
u &= 0 \quad &&\mbox{on} \;\; \partial \Omega, 
\end{alignedat}
\right.
\end{equation}
where $1<p<N-1$, $N<p+p^2$, $r \in (p,p^*-1)$, and $\gamma \in (0,1)$. \\
Problem \eqref{concrete} comes from \eqref{lambdaprob} by choosing, for all $(x,t) \in \Omega \times (0,+\infty)$,
$$a(t):=t^{p-2}\log(1+t), \quad \Phi(t):=\int_{0}^{t}s^{p-1}\log(1+s)\ds, \quad f(x,t):=t^r+t^{-\gamma}.$$
In order to verify the assumptions of Theorem \ref{secondsol}, we explicitly observe that
\begin{itemize}
\item[$1.$] $H_a(t):=\frac{ta'(t)}{a(t)}=p-2+\frac{t}{(t+1)\log(t+1)}$ is a decreasing function in $(0,+\infty)$ with $\displaystyle \lim_{t\to 0^+}H_a(t)=p-1$ and $\displaystyle \lim_{t\to \infty}H_a(t)=p-2$. Then we have
\begin{equation}
\label{indici a}
i_a:=\inf_{t>0}H_a(t)=p-2<p-1=\sup_{t>0}H_a(t)=:s_a.
\end{equation}
%
\item[$2.$] According to De L'H\^opital's rule, $H_{\Phi}(t):=\frac{t\Phi'(t)}{\Phi(t)}$ fulfills
\begin{equation}
\label{nearzero}
\lim_{t\to 0^+}H_{\Phi}(t) = \lim_{t\to 0^+}\frac{ta(t)+t(ta(t))'}{ta(t)} = 2 + \lim_{t \to 0^+} H_a(t) = p+1
\end{equation}
and
\begin{equation}
\label{nearinfinity}
\lim_{t\to \infty}H_{\Phi}(t)=2 + \lim_{t \to \infty} H_a(t)=p.
\end{equation}
\item[$3.$] One has
\begin{equation}
\label{growthindices}
i_a+2\leq i_\Phi\leq s_\Phi\leq s_a+2.
\end{equation}
Indeed, for all $s>0$, we have $i_a\leq \frac{sa'(s)}{a(s)}\leq s_a$. Multiplying by $sa(s)$, an integration by parts in $(0,t)$ gives $i_a\Phi(t)\leq t^2a(t)-2\Phi(t)\leq s_a\Phi(t)$ and our claim follows.
\end{itemize}
From \eqref{indici a}--\eqref{growthindices}, it is readily seen that
$$p=i_a+2\leq i_\Phi\leq p \quad \mbox{and} \quad p+1 \leq s_\Phi\leq s_a+2=p+1,$$
that is,
\begin{equation}
\label{indici Phi}
i_\Phi=p<p+1=s_\Phi.
\end{equation}
Therefore, from \eqref{indici a}, it is clear that ${\rm H(a)_1}$ holds if and only if $p>1$. Bearing in mind \eqref{sobind}, since we have that $s_\Phi=p+1<p^*=i^*_\Phi\leq i_{\Phi_*}$, also ${\rm H(a)_2}$ is verified. On the other hand, ${\rm H(f)_1}$--${\rm H(f)_3}$ follow from Remark \ref{reactrmk} by taking $\Upsilon(t)=t^{r+1}$, being $i_{\Upsilon}=s_{\Upsilon}=r+1$ with $p<r< p^*-1$.
\end{ex}

\begin{rmk}
\label{noregularity}
Regarding Example \ref{ex2}, if we drop the condition $N<p+p^2$ and replace $r \in(p,p^*-1)$ with jointly $p<r$ and $t^r \ll \Phi_*$, we can ensure only that problem \eqref{concrete} admits at least two distinct weak solutions in $W^{1,\Phi}_0(\Omega)$. In particular, two solutions are obtained in the case $p<r\leq p^*-1$, since $t^p \ll \Phi$ forces $t^{r+1} < t^{p^*} \ll \Phi_*$ (with an argument similar to the one in the last part of the proof of Lemma \ref{example}). A similar conclusion holds true for Example \ref{ex1}, replacing $r<q^*$ with $t^r \ll \Phi_*$.
\end{rmk}

\section*{Acknowledgments}

\noindent
The authors wish to thank Prof.~Sunra Mosconi for fruitful discussions about some topics of the present research. \\
The authors are supported by PRIN 2017 `Nonlinear Differential Problems via Variational, Topological and Set-valued Methods' (Grant No. 2017AYM8XW) of MIUR. The second author is also supported by: GNAMPA-INdAM Project CUP$\underline{\phantom{x}}$E55F22000270001; grant `PIACERI 20-22 Linea 3' of the University of Catania. The third author is also supported by the grant `FFR 2021 Roberto Livrea'.

\begin{small}

\end{small}


\begin{thebibliography}{99}
\bibitem{AF}
R.A. Adams and J.F. Fournier, \textit{Sobolev spaces. Second edition}, Pure and Applied Mathematics \textbf{140}, Elsevier/Academic Press, Amsterdam, 2003.
\bibitem{Bo}
G. Bonanno, \textit{A critical point theorem via the Ekeland variational principle}, Nonlinear Anal. \textbf{75} (2012), 2992–3007.
\bibitem{Bo2}
G. Bonanno, \textit{Relations between the mountain pass theorem and local minima}, Adv. Nonlinear Anal. \textbf{1} (2012), 205–220.
\bibitem{BC}
G. Bonanno and P. Candito, \textit{Non-differentiable functionals and applications to elliptic problems with discontinuous nonlinearities}, J. Differential Equations \textbf{244} (2008), 3031–3059.
\bibitem{B}
H. Brezis, \textit{Functional analysis, Sobolev spaces and partial differential equations}, Universitext, Springer, New York, 2011.
\bibitem{Ca1}
S. Campanato, \textit{Proprietà di inclusione per spazi di Morrey}, Ricerche Mat. \textbf{12} (1963), 67–86 (in Italian).
\bibitem{Ca2}
S. Campanato, \textit{Equazioni ellittiche del II° ordine e spazi $\mathscr{L}^{(2,\lambda)}$}, Ann. Mat. Pura Appl. \textbf{69} (1965), 321–381 (in Italian).
\bibitem{CGP}
P. Candito, U. Guarnotta, and K. Perera, \textit{Two solutions for a parametric singular $p$-Laplacian problem}, J. Nonlinear Var. Anal. \textbf{4} (2020), 455–468.
\bibitem{CGS}
M.L.M. Carvalho, J.V.A. Gonçalves, and E.D. da Silva, \textit{On quasilinear elliptic problems without the Ambrosetti-Rabinowitz condition}, J. Math. Anal. Appl. \textbf{426} (2015), 466–483. 
\bibitem{CGSS}
M.L. Carvalho, J.V. Gonçalves, E.D. Silva, and C.A.P. Santos, \textit{A type of Brézis-Oswald problem to $\Phi$-Laplacian operator with strongly-singular and gradient terms}, Calc. Var. Partial Differential Equations \textbf{60} (2021), Paper No. 195, 25 pp.
\bibitem{C}
A. Cianchi, \textit{Hardy inequalities in Orlicz spaces}, Trans. Amer. Math. Soc. \textbf{351} (1999), 2459–2478.
\bibitem{CGHMS}
Ph. Clément, M. Garc\'ia-Huidobro, R. Man\'asevich, and K. Schmitt, \textit{Mountain pass type solutions for quasilinear elliptic equations}, Calc. Var. Partial Differential Equations \textbf{11} (2000), 33–62.
\bibitem{FIN}
N. Fukagai, M. Ito, and K. Narukawa, Positive solutions of quasilinear elliptic equations with critical Orlicz-Sobolev nonlinearity on $\R^N$, Funkcial. Ekvac. \textbf{49} (2006), 235–267.
\bibitem{GG}
L. Gambera and U. Guarnotta, \textit{Strongly singular convective elliptic equations in $\R^N$ driven by a non-homogeneous operator}, Comm. Pure Appl. Math., doi:10.3934/cpaa.2022088.
\bibitem{GST}
J. Giacomoni, I. Schindler, and P. Tak\'a\v{c}, \textit{Sobolev versus H\"{o}lder minimizers and global multiplicity for a singular and quasilinear equation}, Ann. Sc. Norm. Super. Pisa Cl. Sci. (5) 6 (2007), 117–158.
\bibitem{GiaGiu}
M. Giaquinta and E. Giusti, \textit{Global $C^{1,\alpha}$-regularity for second order quasilinear elliptic equations in divergence form}, J. Reine Angew. Math. \textbf{351} (1984), 55–65.
\bibitem{G}
J.-P. Gossez, \textit{Orlicz-Sobolev spaces and nonlinear elliptic boundary value problems}, Nonlinear analysis, function spaces and applications (Proc. Spring School, Horni Bradlo, 1978), pp. 59–94, Teubner, Leipzig, 1979.
\bibitem{GMM}
U. Guarnotta, S.A. Marano, and D. Motreanu, \textit{On a singular Robin problem with convection terms}, Adv. Nonlinear Stud. \textbf{20} (2020), 895–909.
\bibitem{H}
D.D. Hai, \textit{On a class of singular $p$-Laplacian boundary value problems}, J. Math. Anal. Appl. \textbf{383} (2011), 619–626.
\bibitem{KR}
M.A. Krasnosel'ski\u{i} and Ja.B. Ruticki\u{i}, \textit{Convex functions and Orlicz spaces}, P. Noordhoff Ltd., Groningen, 1961.
\bibitem{KJF}
A. Kufner, O. John, S. Fu\v{c}\'ik, \textit{Function spaces}, Monographs and Textbooks on Mechanics of Solids and Fluids, Mechanics: Analysis, Noordhoff International Publishing, Leyden, Academia, Prague, 1977.
\bibitem{LU}
O.A. Ladyzhenskaya and N.N. Ural'tseva, \textit{Linear and quasilinear elliptic equations}, Academic Press, New York-London, 1968.
\bibitem{L}
G.M. Lieberman, \textit{The natural generalization of the natural conditions of Ladyzhenskaya and Ural'tseva for elliptic equations}, Comm. Partial Differential Equations \textbf{16} (1991), 311–361.
\bibitem{MMP}
D. Motreanu, V.V. Motreanu, and N.S. Papageorgiou, \textit{Topological and variational methods with applications to nonlinear boundary value problems}, Springer, New York, 2014.
\bibitem{PapSmy}
N.S. Papageorgiou and G. Smyrlis, \textit{A bifurcation-type theorem for singular nonlinear elliptic equations}, Methods Appl. Anal. \textbf{22} (2015), 147–170.
\bibitem{PS}
P. Pucci and J. Serrin, \textit{The maximum principle}, Prog. Nonlinear Differential Equations Appl. {\bf 73}, Birkh\"auser Verlag, Basel, 2007.
\bibitem{RR}
M.N. Rao and Z.D. Ren, \textit{Theory of Orlicz Spaces}, Marcel Dekker, New York, 1985.
\bibitem{SGC}
C.A. Santos, J.V. Gonçalves, M.L. Carvalho, \textit{About positive $W^{1,\Phi}_{loc}(\Omega)$-solutions to quasilinear elliptic problems with singular semilinear term}, Topol. Methods Nonlinear Anal. \textbf{53} (2019), 491–517.
\bibitem{TF}
Z. Tan and F. Fang, \textit{Orlicz-Sobolev versus Hölder local minimizer and multiplicity results for quasilinear elliptic equations}, J. Math. Anal. Appl. \textbf{402} (2013), 348–370.
%
\end{thebibliography}
\end{document}